\documentclass[10pt,twoside]{amsart}
\usepackage{amsmath,amsfonts,amssymb,amsxtra,setspace,xspace,graphicx,lmodern,psfrag,epsfig,color,latexsym}
\usepackage[T1]{fontenc}
\usepackage[colorlinks=true]{hyperref}\hypersetup{urlcolor=blue, citecolor=red}

\newtheorem{lemma}{Lemma}[section]
\newtheorem{proposition}{Proposition}[section]
\newtheorem{theorem}{Theorem}[section]
\newtheorem{rmk}{Remark}[section]
\newtheorem{cor}{Corollary}[section]

\newcommand{\R}{\mathbb R}
\newcommand{\cc}{\mathbb C}

\begin{document}

\title[On near-cloaking for linear elasticity]{On near-cloaking for linear elasticity}

\bibliographystyle{alpha}

\author{R. Craster}
\address{R.C.: Department of Mathematics, Imperial College London, London, SW7 2AZ, United Kingdom; 
CNRS-Imperial ``Abraham de Moivre'' Unit\'e Mixte Internationale.}
\email{r.craster@imperial.ac.uk}

\author{A. Diatta}
\address{A.D.: Aix-Marseille Universite, CNRS, Centrale Marseille, Institut Fresnel, Avenue Escadrille Normandie-Ni\'emen, 13013, Marseille, France.}
\email{andre.diatta@fresnel.fr}

\author{S. Guenneau}
\address{S.G.: Aix-Marseille Universite, CNRS, Centrale Marseille, Institut Fresnel, Avenue Escadrille Normandie-Ni\'emen, 13013, Marseille, France.}
\email{sebastien.guenneau@fresnel.fr}

\author{H. Hutridurga}
\address{H.H.: Indian Institute of Technology Bombay, Department of Mathematics, Mumbai 400076, India.}
\email{hutri@math.iitb.ac.in}

\date{\today}

\maketitle

\setcounter{tocdepth}{1}
\tableofcontents

\thispagestyle{empty}

\begin{abstract}
We make precise some results on the cloaking of displacement fields in linear elasticity. In the spirit of transformation media theory, the transformed governing equations in Cosserat and Willis frameworks are shown to be equivalent to certain high contrast small defect problems for the usual Navier equations. We discuss near-cloaking for elasticity systems via a regularized transform and perform numerical experiments to illustrate our near-cloaking results. We also study the sharpness of the estimates from [H. Ammari, H. Kang, K. Kim and H. Lee, J. Diff. Eq. 254,  4446-4464 (2013)], wherein the convergence of the solutions to the transmission problems is investigated, when the Lam\'e parameters in the inclusion tend to extreme values. Both soft and hard inclusion limits are studied and we also touch upon the finite frequency case. Finally, we propose an approximate isotropic cloak algorithm for a symmetrized Cosserat cloak.
\end{abstract}

\section{Introduction}
Following the proposal of an electromagnetic invisibility cloak via a geometric transform \cite{pendry, leonhardt}, and inspired by an extension to other equations of physics \cite{Milton_2006}, a number of papers have appeared on possible routes to elastodynamic cloaking \cite{brun,Norris_2011,Norris_2012,diatta-guenneau-sphericalcloak,colquitt2013}. Potential applications of transformation elastodynamics are in the protection of specific regions from mechanical vibrations ranging from ultrasonics \cite{colombi2017} to civil engineering \cite{brule-enoch-guenneau-prl2014} and geophysics \cite{colombi,roux2018}.

The theory of cloaking has an intricate connection to the theory of inverse problems in the context of partial differential equations (PDEs). Loosely speaking, a successful cloak should make the reconstruction of material properties based on boundary measurements impossible. Greenleaf and coauthors have shown in \cite{greenleaf} that two different conductivity distributions can lead to the same boundary voltage measurements. This result was achieved using a singular geometric transform, which blows up a point onto a ball. Incidentally, the transformation appearing in \cite{pendry} coincides with that of \cite{greenleaf}, although the latter work did not make connection to invisibility cloaks.

Pendry's construction of a perfect cloak in \cite{pendry} essentially exploits the transformation invariance property of the Maxwell system \cite{Dolin1961, Nicolet_1994, Ward_1996}. This invariance property is shared by electrostatic, acoustic and thermostatic problems. Unfortunately, elastostatic systems do not retain their form under change-of-variables \cite{Milton_2006}. It should be noted that the geometric transform used in \cite{pendry} is singular at the inner boundary of the cloak. This leads to technical difficulties in the mathematical analysis of transformed models. This was remedied in \cite{Kohn_2008} by the introduction of a regularized geometric transform, which blows up a small ball onto the cloaked region. Here, the authors introduced the notion of near-cloaking, as opposed to perfect cloaking of \cite{pendry}.  An important observation in \cite{Kohn_2008} is that the concept of near-cloaking is related to the study of certain boundary value problems with small volume inhomogeneities. The study of PDEs in the presence of small volume inhomogeneities is a well-established topic of interest amongst analysts. One can consult \cite{Friedman_1989} for the conductivity problem, \cite{kohn2010} for the Helmholtz problem, \cite{Ammari_2013} and \cite[chapter 2]{Ammari_elasticity_book} for linear elasticity systems. We also refer the reader to the review paper \cite{greenleaf2009} on cloaking by change-of-variable based methods.

In recent years, there has been considerable interest in the applied mathematics literature on the cloaking of elastic displacement fields. For instance, in \cite{hu2015}, the authors develop a regularized approximate cloak involving a lossy layer for the Lam\'e system and they derive sharp asymptotic estimates to assess the cloaking performance at finite frequencies. This approach borrows ideas from \cite{kohn2010} where a similar construction was proposed to cloak acoustic waves at finite frequencies.  In another interesting approach \cite{abbas2017}, the authors make use of layer potential techniques and transformation matrix (so-called T-matrix) to achieve an enhancement of near cloaking of elastic waves in two-dimensions, in the low-frequency regime. A drawback in adapting the strategy in \cite{hu2015} to the static case is that estimates therein blowup in the zero-frequency limit, whereas the strategy developed in \cite{abbas2017} through the construction of the so-called S-vanishing structures depends upon the Lam\'e parameters of the inclusion, see \cite[section 5]{abbas2017} for further details. Also note that there has been a recent near-cloaking theory at finite frequencies for the Cosserat cloak, but with a lossy layer \cite{lin2018}, and estimates therein do not encompass the static case.

Here, we retain the classical transformation based cloaking strategy as proposed in \cite{Milton_2006,diatta-guenneau-impedancecloak}: our inspiration comes from the work of \cite{Kohn_2008} in electrostatics. There Kohn and co-authors linked the transformation based cloaking to the study of a Poisson problem with a small volume defect, then applied the result of Friedman and Vogelius \cite{Friedman_1989}, which established the difference between the solutions to the homogeneous and the perturbed problems to be a monotone decreasing function of the size of the defect, irrespective of the contrast in conductivity. Hence, the work of Ammari and coauthors \cite{Ammari_2013} on the study of transmission problems in linear elasticity systems plays a prominent role in our near cloaking approach. Nonetheless, their asymptotic results are valid subject to a certain upper bound on the bulk modulus of the inclusion. A straightforward application of their result to deduce a near-cloaking result would involve imposition of similar conditions on the bulk modulus that could be cloaked; this is in contrast to the conductivity case. Hence we improve the result of \cite{Ammari_2013} by lifting the boundedness assumption on the bulk modulus of the inclusion, thus proving a near-cloaking result irrespective of the material properties in the cloaked region.

This paper is organised as follows: section \ref{sec:math-setting} introduces the mathematical model, makes precise our notion of near-cloaking, gathers some basic qualitative properties of elasticity tensors and details transmission problems in high-contrast elastic media. Transformation based cloaking in linear elasticity is discussed in section \ref{sec:transformation-elasticity}, where we treat both the Willis and Cosserat settings. Section \ref{sec:near-cloaking-results} deals with the near-cloaking results. Our main results are Theorem \ref{thm:near-cloak-2D-cosserat} in the Cosserat case and Theorem \ref{thm:near-cloak-2D-3D-willis} in the Willis case. In section \ref{sec:symmetrize}, we propose a symmetrization procedure for the Cosserat cloak followed by a construction of an isotropic approximation to that symmetrized tensor using homogenization techniques. Finally, in section \ref{sec:numerics} we present our numerical results in two dimensions, where we illustrate near-cloaking in Cosserat setting and the asymptotic behaviour of solutions to the transmission problems in both the soft and hard inclusion limits. 

\section{Mathematical setting}\label{sec:math-setting}
Throughout this work, $d$ stands for the dimension of the Euclidean space $\R^d$. In a linear elastic material with fourth order constitutive tensor $\cc$, occupying a bounded Lipschitz domain $\Omega\subset \R^d$ ($d=2,3$), the partial differential equation in \emph{elastostatics} for the unknown displacement field $\mathbf{u}:\Omega\to\R^d$ reads
\begin{equation}\label{eq:navier-general}
\left\{
\begin{aligned}
{\mathbf \nabla}\cdot \big(\cc: \nabla {\mathbf u} \big) & = \mathbf{0} \qquad \mbox{ in }\Omega,
\\
\big(\cc: \nabla {\mathbf u}\big)\mathbf{n} & = \mathbf{g} \qquad \mbox{ on }\partial \Omega,
\end{aligned}\right.
\end{equation}
where ${\big(\cc: \nabla {\mathbf u} \big)}_{ij}:=\sum_{k,l=1}^d\cc_{ijkl}{\big(\nabla {\mathbf u} \big)}_{kl}$ and the traction (Neumann) datum
$\mathbf{g}\in[\mathrm H^{-\frac{1}{2}}(\partial\Omega)]^d$ satisfies the compatibility condition
\[
\int_{\partial\Omega} \mathbf{g}(x) \, {\mathrm d}\sigma(x) = \mathbf{0}.
\]  
Here $\mathbf{n}(x)$ denotes the outward unit normal to $\partial\Omega$ at point $x=(x_1,\dots,x_d)$. We assume that the elasticity tensor $\mathbb{C}=\left(C_{ijkl}\right)$ in \eqref{eq:navier-general} satisfies the ellipticity condition, i.e., there exists a constant $\mathfrak{a}>0$ such that
\begin{align}\label{eq:ellipticity-condition}
(\mathbb{C}:A):A = \sum_{i,j,k,l=1}^d C_{ijkl} A_{ij} A_{kl} \ge \mathfrak{a} \sum_{i,j=1}^d \left\vert A_{ij} \right\vert^2 =: \mathfrak{a} \left\Vert A \right\Vert^2,
\end{align}
for any symmetric second order tensor $A=\big(A_{ij}\big)$. It is a classical matter in this setting to show that for any given $\mathbf{g}\in[\mathrm H^{-\frac{1}{2}}(\partial\Omega)]^d$, there exists a unique solution ${\mathbf u}\in [\mathrm H^1(\Omega)]^d$ to the system \eqref{eq:navier-general} -- we refer the reader to, e.g., \cite{Ciarlet_1988}. Note that the solution $\mathbf{u}(x)$ is unique up to addition of constants. Hence, we can obtain uniqueness by imposing a normalisation
\[
\int_\Omega {\mathbf u}(x)\, {\mathrm d}x = \mathbf{0}.
\]
As the partial differential equation \eqref{eq:navier-general} is well-posed, the following Neumann-to-Dirichlet map (NtD for short) is well-defined:
\begin{align*}
\Lambda_{\mathbb{C}} : [\mathrm H^{-\frac{1}{2}}(\partial\Omega)]^d & \to [\mathrm H^{\frac{1}{2}}(\partial\Omega)]^d
\\
\mathbf{g} & \mapsto {\mathbf u}|_{\partial\Omega},
\end{align*}
where $\mathbf{u}(x)$ is the unique solution to \eqref{eq:navier-general} corresponding to the traction boundary datum $\mathbf{g}(x)$. There also exists a variational characterisation of the displacement field ${\mathbf u}(x)$ which solves the elastostatic problem \eqref{eq:navier-general}:
\begin{align}\label{eq:variational-form}
\min_{{\mathbf v}\in\mathcal{V}^\mathbf{g}_{\cc}} \int_\Omega \big( \cc:{\mathbf e}({\mathbf v}) \big) : {\mathbf e}({\mathbf v}) \, {\mathrm d}x,
\end{align}
with ${\mathbf e}({\mathbf v})=\frac12\left( \nabla {\mathbf v} + \nabla {\mathbf v}\!^\top \right)$, the symmetrized gradient, where $B\!^\top$ stands for the transpose of $B$. The admissible function space $\mathcal{V}^\mathbf{g}_{\cc}$ in \eqref{eq:variational-form} is defined below
\begin{align}\label{eq:variational-form-admissible-space}
\mathcal{V}^\mathbf{g}_{\cc} := \left\{ {\mathbf w}\in [\mathrm H^1(\Omega)]^d \mbox{ such that } \big( \cc : {\mathbf e}({\mathbf w}) \big) {\mathbf n} = \mathbf{g}  \mbox{ on }\partial\Omega \right\}.
\end{align}
The constitutive tensor $\cc$ of a typical elastic material would satisfy the following symmetries:
\begin{align}\label{toto}
\mbox{major symmetry: }C_{ijkl} = C_{klij},
\qquad
\mbox{ minor symmetries: }C_{ijkl} = C_{jikl} = C_{ijlk},
\end{align}
for $i,j,k,l=1,..., d$. 
However, as shown in \cite{brun, diatta-guenneau-sphericalcloak}, transformation elasticity may result in exotic materials with anisotropic inhomogeneous elasticity tensors without the minor symmetries (see section \ref{sec:transformation-elasticity} for details). Such materials also called Cosserat materials by the authors \cite{brun, diatta-guenneau-sphericalcloak}, as a generalization of the known Cosserat materials, still have isotropic but inhomogeneous mass densities. In an endeavor to enforce the symmetry of all tensors, Milton et al. \cite{Milton_2006} propose the so-called Willis equation that produce materials (Willis materials) with fully symmetric elasticity tensor. However, such Willis materials possess additional 3-tensors and display anisotropic inhomogeneous mass densities.
In isotropic homogeneous elastic media, the elasticity tensor $\cc=\left(C_{ijkl}\right)$ is fully symmetric. More precisely, its components read
\begin{align}\label{eq:isotropic-elasticity-tensor}
C_{ijkl} = \lambda\, \delta_{ij}\delta_{kl} + \mu \left( \delta_{ik} \delta_{jl} + \delta_{il} \delta_{jk} \right),
\end{align}
where the standard Kr\"onecker delta notation $\delta_{ij}$ is used  and it satisfies Hooke's law: 
\begin{align}\label{eq:hooke}
\sigma = \cc:  \epsilon \;,
\end{align} 
where $\sigma$ and $\epsilon$ are the stress and the strain tensors, respectively. The constants $\lambda$ and $\mu$ in \eqref{eq:isotropic-elasticity-tensor} are called the bulk modulus (first Lam\'e parameter) and the shear modulus respectively. The ellipticity condition \eqref{eq:ellipticity-condition} for the isotropic elasticity tensor \eqref{eq:isotropic-elasticity-tensor} translates as
\begin{align}\label{eq:strong-convexity-Lame-parameters}
\mu > 0, \qquad d\lambda + 2 \mu >0.
\end{align}
The parameters $\lambda$, $\mu$ help us define the Poisson ratio $\nu$, the Bulk modulus $\kappa$ and the Young's modulus $E$ as
\begin{align}\label{eq:poisson-bullk-young-modulii-defn}
\nu := \frac{\lambda}{(d-1)\lambda+2\mu},
\qquad
\kappa := \frac{d\lambda + 2\mu}{d},
\qquad
E := \frac{2\mu(d\lambda+2\mu)}{(d-1)\lambda+2\mu}.
\end{align}
We note that these assumptions ensure that $0<\nu<0.5$, but within the theory of composites it has been shown that $\nu$ can get close to $-1$ \cite{milton1992}. In isotropic setting, the elasticity system \eqref{eq:navier-general} takes the form
\begin{equation}\label{eq:navier-isotropic}
\left\{
\begin{aligned}
\mu\, \Delta {\mathbf u} + (\lambda+\mu)\nabla (\nabla .{\mathbf u}) & = \mathbf{0} \qquad \mbox{ in }\Omega,
\\
\lambda\, (\nabla .{\mathbf u}) {\mathbf n} + \mu\, {\mathbf e}({\mathbf u}) {\mathbf n} & = \mathbf{g} \qquad \mbox{ on }\partial \Omega.
\end{aligned}\right.
\end{equation}
We further note that well posedness of (\ref{eq:navier-general}) holds even for tensors with minor symmetry breaking in the elasticity tensor, see e.g. \cite{Dyszlewicz_2004}.

\subsection{Notion of near-cloaking}
Here we make precise the notion of near-cloaking for elastic systems. Given a region $\mathcal{Q}$ strictly contained in $\Omega$, the goal is to surround $\mathcal{Q}$ by an elastic material such that any linear elastic object (inclusion) in $\mathcal{Q}$ goes nearly undetected by measurements on the boundary of $\Omega$. More precisely, let the background be occupied by a homogeneous isotropic material with Lam\'e parameters $\lambda_0$, $\mu_0$ satisfying \eqref{eq:strong-convexity-Lame-parameters}. For any given positive $\delta$, the aforementioned construction of the surrounding elastic material (possibly depending on $\delta$) should be such that the associated displacement field is $\delta$-close (in some chosen topology) to that of the homogeneous isotropic material (with the Lam\'e parameters $\lambda_0$, $\mu_0$) as measured at the boundary of $\Omega$, irrespective of the elastic property of the inclusion in $\mathcal{Q}$. Throughout, we use the notation $B_r$ to denote a Euclidean ball of radius $r=\sqrt{x_1^2+\dots+x_d^2}$ centred at the origin. To fix ideas, we take the region $\mathcal{Q}$ to be $B_1$ and the surrounding region to be the annulus $B_2\setminus B_1$. 
Given a number $0<\delta<1$, our objective is to give a recipe for constructing a fourth order tensor $\cc^\delta_{\mathrm{cl}}(x)$ with the structure
\begin{equation}\label{eq:goal:cloak-defn}
\cc^\delta_{\mathrm{cl}}(x) =
\left\{
\begin{array}{ll}
\cc_0 & \quad \mbox{ for }x\in\Omega\setminus B_2,\\[0.2 cm]
\mathbb{D}^\delta & \quad \mbox{ for }x\in B_2\setminus B_1,\\[0.2 cm]
\cc_1 & \quad \mbox{ for }x\in B_1,
\end{array}\right.
\end{equation}
where $\cc_0$ and $\cc_1$ are isotropic fourth order tensors of the form \eqref{eq:isotropic-elasticity-tensor} with Lam\'e parameters $\lambda_0,\mu_0$ and $\lambda_1, \mu_1$ respectively (both pairs satisfying the strong convexity condition \eqref{eq:strong-convexity-Lame-parameters}). The above construction should be such that the solution $\mathbf{u}_\mathrm{cl}^\delta$ to \eqref{eq:navier-general} with $\cc_\mathrm{cl}^\delta$ as the elasticity tensor satisfies
\begin{align}\label{eq:goal:distance-H12-norm}
\left\Vert \mathbf{u}^\delta_\mathrm{cl} - \mathbf{u}_\mathrm{hom} \right\Vert_{[\mathrm{H}^\frac{1}{2}(\partial\Omega)]^d} \lesssim h(\delta),
\end{align}
where $\lim_{\delta\to 0}h(\delta)=0$ and $\mathbf{u}_\mathrm{hom}(x)$ is the solution to the isotropic system \eqref{eq:navier-isotropic} with Lam\'e parameters $\lambda_0$, $\mu_0$. Furthermore, the closeness \eqref{eq:goal:distance-H12-norm} should hold irrespective of the values of the Lam\'e parameters $\lambda_1$, $\mu_1$ in $B_1$ as long as they satisfy the convexity condition \eqref{eq:strong-convexity-Lame-parameters}.

\subsection{Ordering property of elasticity tensors}
In this subsection, we recall and prove some qualitative and quantitative results on elasticity systems \cite{diatta-guenneau-impedancecloak}.
We say that two fourth order tensors $\mathbb{C}_1$ and $\mathbb{C}_2$ are ordered if we have
\begin{align}\label{eq:C-1_C-2_ordered-defn}
\big(\mathbb{C}_1:A\big):A \le \big(\mathbb{C}_2:A\big):A ,
\end{align}
for any arbitrary symmetric second order tensor $A$. If $\cc_1$ and $\cc_2$ are isotropic homogeneous tensors, then there is a characterisation of the aforementioned ordering property in terms of their Lam\'e parameters. This is the purpose of the next result. 
\begin{lemma}\label{lem:second-inequality-tensors}
Let $\cc_1$ and $\cc_2$ be two isotropic homogeneous fourth order tensors with Lam\'e coefficients $\lambda_i$ and $\mu_i,$ $i=1,2,$ respectively. Then $\cc_1$ and $\cc_2$ are ordered as in \eqref{eq:C-1_C-2_ordered-defn} if and only if $d\lambda_1+2\mu_1 \leq d\lambda_2+2\mu_2$ and  $\mu_1\leq\mu_2.$
\end{lemma}

\proof  For a symmetric tensor $A$ of order 2, we have
\begin{align}\label{eq:ordered-proof-intermediate}
(\cc_1:A ):A =\frac{1}{d} (d\lambda_1+2\mu_1)\displaystyle\left(\sum_{i=1}^d A_{ii}\right)^2 + 4\mu_1\displaystyle\sum_{1\le i<j\le d}A_{ij}^2+ \frac{2\mu_1}{d}\displaystyle\sum_{1\le i<j\le d}\left(A_{ii}-A_{jj}\right)^2.
\end{align}
A similar expression holds for $(\cc_2:A):A$. Hence we have
\begin{eqnarray}
(\cc_2:A ):A -(\cc_1:A ):A &=& \frac{1}{d} \Big(d\lambda_2+2\mu_2-(d\lambda_1+2\mu_1)\Big)\displaystyle\left(\displaystyle\sum_{i=1}^d A_{ii}\right)^2 \nonumber \\
&+& 4(\mu_2-\mu_1)\displaystyle \sum_{1\leq i <j\leq d} A_{ij}^2+ \frac{2}{d}(\mu_2-\mu_1)\displaystyle \sum_{1\leq i <j\leq d} \left(A_{ii}-A_{jj}\right)^2. \nonumber
\end{eqnarray}
By inspection, it follows that the conditions $d\lambda_1+2\mu_1 \leq d\lambda_2+2\mu_2$ and $\mu_1\leq\mu_2$ suffice to guarantee that $\cc_1$ and $\cc_2$ are ordered. Now, in order to show that they are indeed necessary, let us suppose that $\mathbb C_1$ and $\mathbb C_2$ are ordered as in \eqref{eq:C-1_C-2_ordered-defn} and further choose $A$ to be the second order identity tensor which yields $d\lambda_1+2\mu_1 \leq d\lambda_2+2\mu_2$. Next, by taking $A$ to be a symmetric second order tensor with vanishing diagonal elements, we obtain $\mu_1\leq\mu_2$.
\qed

The result we state next gives precise lower and upper bounds on the fourth order isotropic homogeneous tensors in terms of their Lam\'e coefficients. This observation will be useful while performing the soft and hard inclusion asymptotics in the transmission problems addressed later on in this paper.
\begin{lemma}\label{lem:lower-upper-bound-isotropic-tensor}
Let $\cc$ be a homogeneous isotropic fourth order tensor with Lam\'e coefficients $\lambda$ and $\mu$, then, for every symmetric second order tensor $A$, we have 
\begin{align}\label{eq:upper-lower-bounds-isotropic-tensor}
C_1\left\Vert A \right\Vert^2 \le \big( \cc: A\big) : A \le C_2 \left\Vert A \right\Vert^2,
\end{align}
with the constants $C_1$ and $C_2$ given by
\[
C_1 := \min\{d\lambda+2\mu,2\mu \}
\quad
\mbox{ and }
\quad
C_2 := \max\{d\lambda+2\mu,2\mu \}.
\]
\end{lemma}
\proof
For a symmetric second order tensor $A$,
\begin{eqnarray} 
\left\Vert A \right\Vert^2 =\frac{1}{d}\displaystyle\left(\sum_{i=1}^dA_{ii}\right)^2 + 2\displaystyle\sum_{i<j}A_{ij}^2+ \frac{1}{d}\displaystyle\sum_{i<j}\left(A_{ii}-A_{jj}\right)^2. \nonumber
\end{eqnarray} 
Hence the estimate \eqref{eq:upper-lower-bounds-isotropic-tensor} follows from \eqref{eq:ordered-proof-intermediate}.
\qed

The next result says that the ordering property of the elasticity tensors will be picked up by the corresponding NtD maps.
\begin{lemma}\label{lem:ordering-NtD-maps}
Let $\cc_1$ and $\cc_2$ be fourth order elasticity tensors satisfying the symmetry assumptions (\ref{toto}) and the ellipticity condition \eqref{eq:ellipticity-condition}. Further assume that they are ordered in the sense of \eqref{eq:C-1_C-2_ordered-defn}. Then the corresponding NtD maps are ordered in the same manner, more precisely, we have
\begin{align}
\left\langle \Lambda_{\cc_1}\mathbf{f}, \mathbf{f} \right\rangle \le \left\langle \Lambda_{\cc_2}\mathbf{f}, \mathbf{f} \right\rangle
\qquad
\mbox{ for any }\mathbf{f}\in [\mathrm H^{-\frac12}(\partial\Omega)]^d,
\end{align}
where $\langle \cdot, \cdot \rangle$ denotes the duality pairing between $[\mathrm H^{\frac12}(\partial\Omega)]^d$ and $[\mathrm H^{-\frac12}(\partial\Omega)]^d$.
\end{lemma}

\proof
The proof goes via the energy method. In the elasticity system \eqref{eq:navier-general}, take the elasticity tensor to be $\cc_1$ and the traction datum to be some $\mathbf{f}\in [\mathrm H^{-\frac12}(\partial\Omega)]^d$. Multiply the equation in \eqref{eq:navier-general} by the displacement field $\mathbf u(x)$ and integrate over the medium $\Omega$ yielding
\[
\left\langle \Lambda_{\cc_1}\mathbf{f}, \mathbf{f} \right\rangle = \int_\Omega \big( \cc_1:{\mathbf e}({\mathbf u}) \big) : {\mathbf e}({\mathbf u}) \, {\mathrm d}x \le \int_\Omega \big( \cc_1:{\mathbf e}({\mathbf v}) \big) : {\mathbf e}({\mathbf v}) \, {\mathrm d}x,
\]
where the inequality on the right follows from the variational formulation \eqref{eq:variational-form}. We further continue the above inequality by exploiting the ordering assumption made on the tensors $\cc_1$ and $\cc_2$ to get
\[
\int_\Omega \big( \cc_1:{\mathbf e}({\mathbf v}) \big) : {\mathbf e}({\mathbf v}) \, {\mathrm d}x \le \int_\Omega \big( \cc_2:{\mathbf e}({\mathbf v}) \big) : {\mathbf e}({\mathbf v}) \, {\mathrm d}x.
\] 
As a final step, we take ${\mathbf v}(x)$ to be, in particular, the displacement field associated with the elasticity tensor $\cc_2$ and with the same traction $\mathbf{f}(x)$ at the boundary, resulting in
\[
\langle \Lambda_{\cc_1}\mathbf{f}, \mathbf{f}\rangle \le \int_\Omega \big( \cc_2:{\mathbf e}({\mathbf v}) \big) : {\mathbf e}({\mathbf v}) \, {\mathrm d}x = \langle \Lambda_{\cc_2}\mathbf{f}, \mathbf{f}\rangle.
\]
\qed

\subsection{Transmission problems in high contrast media}
In this part of the paper, we study transmission problems in the context of linear isotropic elasticity. More specifically, we consider the Lam\'e parameters of the background $\Omega\setminus D$ to be $\lambda_0, \mu_0$ and those of the inclusion $D$ to be $\lambda_1, \mu_1$. As usual, each of these Lam\'e parameter pairs are assumed to satisfy the strong convexity condition \eqref{eq:strong-convexity-Lame-parameters}. The elastic inclusion $D$ is assumed to be a bounded Lipschitz domain. We consider the following transmission problem:
\begin{equation}\label{eq:navier-transmission}
\left\{
\begin{aligned}
\mu_0\, \Delta {\mathbf u} + (\lambda_0+\mu_0)\nabla \big(\nabla .{\mathbf u} \big) & = \mathbf{0} & \qquad \mbox{ in }\Omega\setminus D,
\\
\mu_1\, \Delta {\mathbf u} + (\lambda_1+\mu_1)\nabla \big( \nabla .{\mathbf u} \big) & = \mathbf{0} & \qquad \mbox{ in }D,
\\
{\mathbf u}|_- & = {\mathbf u}|_+ & \qquad \mbox{ on }\partial D,
\\
\frac{\partial{\mathbf u}}{\partial \mathbf{m}}\Big|_- & = \frac{\partial{\mathbf u}}{\partial \mathbf{m}}\Big|_+ & \qquad \mbox{ on }\partial D,
\\
\lambda_0\, (\nabla .{\mathbf u}) {\mathbf n} + \mu_0\, {\mathbf e}({\mathbf u}) {\mathbf n} & = \mathbf{g} & \qquad \mbox{ on }\partial \Omega,
\end{aligned}\right.
\end{equation}
where $\mathbf{m}(x)$ is the outward unit normal to $\partial D$ at $x$. We take the traction datum $\mathbf{g}\in [\mathrm H^{-\frac{1}{2}}(\partial\Omega)]^d$. Here we have used the following notations for the trace and the Fr\'echet derivative at the inner boundary
\begin{equation*}
\left.
\begin{aligned}
{\mathbf u}|_\pm(x) & := \lim_{s\to0} {\mathbf u}\left(x\pm s\, \mathbf{m}(x)\right) 
\\
\frac{\partial{\mathbf u}}{\partial \mathbf{m}}\Big|_-(x) & := \lim_{s\to0} \Big\{ \lambda_1 \big( \nabla \cdot {\mathbf u} \left(x-s\mathbf{m}(x)\right) \big)\mathbf{m}(x) + \mu_1\, \big( {\mathbf e}({\mathbf u}) \left(x-s\mathbf{m}(x)\right) \big)\mathbf{m}(x) \Big\}
\\
\frac{\partial{\mathbf u}}{\partial \mathbf{m}}\Big|_+(x) & := \lim_{s\to0} \Big\{ \lambda_0 \big( \nabla \cdot {\mathbf u} \left(x+s\mathbf{m}(x)\right) \big)\mathbf{m}(x) + \mu_0\, \big( {\mathbf e}({\mathbf u}) \left(x+s\mathbf{m}(x)\right) \big)\mathbf{m}(x) \Big\}
\end{aligned}\right\}
\,\mbox{ for }x\in\partial D.
\end{equation*}
We are interested in understanding the solution ${\mathbf u}(x)$ to the transmission problem \eqref{eq:navier-transmission} when the Lam\'e parameters in the inclusion $D$ show high contrast with respect to their values elsewhere in the domain. From Lemma \ref{lem:lower-upper-bound-isotropic-tensor}, it follows that this high contrast setting corresponds to studying the regimes: 
\begin{align*}
\mbox{\bf Hard inclusions: } \qquad \min\{d\lambda_1+2\mu_1, \mu_1 \} & \to \infty,
\\
\mbox{\bf Soft inclusions: } \qquad \max\{d\lambda_1+2\mu_1, \mu_1 \} & \to0.
\end{align*}
The limit problem corresponding to the soft inclusions is
\begin{equation}\label{eq:navier-transmission-soft-inclusion}
\left\{
\begin{aligned}
\mu_0\Delta {\mathbf u}_0 + (\lambda_0+\mu_0)\nabla \big(\nabla .{\mathbf u}_0\big) & = \mathbf{0} & \qquad \mbox{ in }\Omega\setminus D,
\\
\frac{\partial{\mathbf u}_0}{\partial \mathbf{m}}\Big|_+ & = \mathbf{0} & \qquad \mbox{ on }\partial D,
\\
\lambda_0\, (\nabla .{\mathbf u}_0) {\mathbf n} + \mu_0\, {\mathbf e}({\mathbf u}_0) {\mathbf n} & = \mathbf{g} & \qquad \mbox{ on }\partial \Omega.
\end{aligned}\right.
\end{equation}
The limit problem corresponding to the hard inclusions is
\begin{equation}\label{eq:navier-transmission-hard-inclusion}
\left\{
\begin{aligned}
\mu_0\Delta {\mathbf u}_\infty + (\lambda_0+\mu_0)\nabla (\nabla .{\mathbf u}_\infty) & = \mathbf{0} & \qquad \mbox{ in }\Omega\setminus D,
\\
{\mathbf u}_\infty & = \sum_{j=1}^{d(d+1)/2} \alpha_j \pmb{\Psi}_j & \qquad \mbox{ on }\partial D,
\\
\lambda_0\, (\nabla .{\mathbf u}_\infty) {\mathbf n} + \mu_0\, {\mathbf e}({\mathbf u}_\infty) {\mathbf n} & = \mathbf{g} & \qquad \mbox{ on }\partial \Omega,
\end{aligned}\right.
\end{equation}
where $\pmb{\Psi}_j$, $j=1,\dots,d(d+1)/2$ forms a set of basis functions in the space of rigid displacements in $D$. Furthermore, the coefficients $\alpha_j$ in the second line of \eqref{eq:navier-transmission-hard-inclusion} are determined by the orthogonality condition:
\begin{align}\label{eq:hard-inclusion-orthogonality-condition}
\int_{\partial D} \frac{\partial{\mathbf u}_\infty}{\partial \mathbf{m}}\Big|_+(x) \cdot \pmb{\Psi}_j(x) \, {\mathrm d}\sigma(x) = 0
\qquad \mbox{ for }j= 1, \dots, \frac{d(d+1)}{2}. 
\end{align}
Next we recall a couple of asymptotic results from \cite[Theorem 4.2, page 4458]{Ammari_2013}.
\begin{theorem}\label{thm:extreme-zero}
Let ${\mathbf u}_\mathrm{tr}(x)$ denote the solution to the transmission problem \eqref{eq:navier-transmission} and let ${\mathbf u}_0(x)$ be the solution to the soft inclusion problem \eqref{eq:navier-transmission-soft-inclusion}. Then there exists a constant $\gamma>0$ such that 
\begin{align}
\left\Vert {\mathbf u}_\mathrm{tr} - {\mathbf u}_0 \right\Vert_{[\mathrm H^1(\Omega\setminus D)]^d} \lesssim (\kappa_1+\mu_1)^{\frac14}
\end{align}
for all bulk and shear moduli $\kappa_1,\mu_1\le \gamma$ inside the inclusion $D$.
\end{theorem}
\begin{theorem}\label{thm:extreme-infty}
Let ${\mathbf u}_\mathrm{tr}(x)$ denote the solution to the transmission problem \eqref{eq:navier-transmission} and let ${\mathbf u}_\infty(x)$ be the solution to the hard inclusion problem \eqref{eq:navier-transmission-hard-inclusion}. Let $\beta>0$ be a given constant. Then there exist constants $M, \tilde{\mu} >0$ such that 
\begin{align}
\left\Vert {\mathbf u}_\mathrm{tr} - {\mathbf u}_\infty \right\Vert_{[\mathrm H^1(\Omega)]^d} \le\frac{M}{\sqrt{\mu}}
\end{align}
for all first Lam\'e parameters $\lambda_1\le \beta$ and shear moduli $\mu\ge \tilde{\mu}$ inside the inclusion $D$.
\end{theorem}
A more general convergence result in the hard inclusions case has been recently proved in \cite{Bao_2015} which, unlike Theorem \ref{thm:extreme-infty}, does not impose any boundedness assumption on the first Lam\'e parameter $\lambda_1$ in the inclusion. However, this generalisation in \cite{Bao_2015} does not give, unlike Theorem \ref{thm:extreme-infty}, a rate of convergence in terms of the Lam\'e parameters in the inclusion. The proofs of Theorem \ref{thm:extreme-zero} and Theorem \ref{thm:extreme-infty} employ layer potential techniques for Lam\'e systems, whereas the results in \cite{Bao_2015} are obtained via variational characterisation. We will return to these high contrast transmission problems in section \ref{sec:numerics} where we report on some of our numerical experiments in the context of $2D$ elasticity. 

\section{Transformation elasticity}\label{sec:transformation-elasticity}
As noticed in \cite{Milton_2006}, the change-of-variable method needs to be constrained by a specific gauge proportional to the Jacobian of the transformation in order to ensure symmetry of the transformed elasticity tensor. The price to pay for preserving all the symmetries of the transformed order-4 tensor is the appearance of additional lower order tensors in the transformed equation. An alternative approach was proposed in \cite{brun,diatta-guenneau-sphericalcloak}, which preserves the structure of the Navier system, at the cost of losing minor symmetries of the elasticity tensor. The transformed materials are termed Willis media in the former case, and Cosserat media in the latter.  

The following is the change-of-variable principle in the Cosserat setting.
\begin{proposition}\label{prop:change-variable-cosserat}
Let $\cc$ be a fourth order tensor. Consider a Lipschitz invertible map $\mathbb{F}:\Omega\mapsto\Omega$ such that $\mathbb{F}(x) = x$ for each $x\in \Omega\setminus B_2$. Furthermore, assume that the associated Jacobians satisfy ${\mathrm det}(D\mathbb{F})(x), {\mathrm det}(D\mathbb{F}^{-1})(x) \ge M$, for some $M >0$ and for a.e. $x\in\Omega$. Then ${\mathbf u}(x)$ is a solution to 
\begin{align*}
\nabla \cdot \Big(\cc(x): \nabla {\mathbf u}\Big) = \mathbf{0} \qquad \mbox{ in }\Omega
\end{align*}
if and only if ${\mathbf v} = {\mathbf u}\circ \mathbb{F}^{-1}$ is a solution to
\begin{align}\label{cosserattransformed}
\nabla \cdot \Big(\mathbb{F}^*\cc(y): \nabla {\mathbf v}\Big) = \mathbf{0} \qquad \mbox{ in }\Omega,
\end{align}
where the tensor $\mathbb{F}^*\cc$ is defined as
\begin{equation}\label{eq:push-forward-formula}
\begin{aligned}
\left[ \mathbb{F}^*\cc \right]_{ijkl}(y) = \frac{1}{{\mathrm det }(D\mathbb{F})(x)} \sum_{p,q=1}^d \frac{\partial \mathbb{F}_i}{\partial x_p}(x) \frac{\partial \mathbb{F}_k}{\partial x_q}(x) \cc_{pjql}(x),
\end{aligned}
\end{equation}
for all $i,j,k,l\in\{1,\dots,d\}$, with the understanding that the right hand sides in \eqref{eq:push-forward-formula} are computed at $x=\mathbb{F}^{-1}(y)$. Moreover we have,
\begin{align}\label{eq:prop:change-variable-assertion}
{\mathbf u}(x) = {\mathbf v}(x) \qquad \mbox{ for all }x\in\Omega\setminus B_2.
\end{align}
\end{proposition}
The following is the change-of-variable principle in the Willis setting.
\begin{proposition}\cite[Proposition B.1]{Milton_2006}\label{prop:change-variable-willis}
Let $\cc$ be a fourth order tensor. Consider a Lipschitz invertible map $\mathbb{F}:\Omega\mapsto\Omega$ such that $\mathbb{F}(x) = x$ for each $x\in \Omega\setminus B_2$. Furthermore, assume that the associated Jacobians satisfy ${\mathrm det}(D\mathbb{F})(x), {\mathrm det}(D\mathbb{F}^{-1})(x) \ge M >0$ for a.e. $x\in\Omega$. Then ${\mathbf u}(x)$ is a solution to 
\begin{align*}
\nabla \cdot \Big(\cc(x): \nabla {\mathbf u}\Big) = \mathbf{0} \qquad \mbox{ in }\Omega
\end{align*}
if and only if
\begin{align}\label{gauge}
{\mathbf v} = \left(\left[D\mathbb{F}^\top\right]^{-1}{\mathbf u}\right)\circ \mathbb{F}^{-1}
\end{align}
is a solution to
\begin{align}\label{eq:willis-model-change-of-variable-princ}
\nabla \cdot \Big(\mathbb{F}_1^*\cc(y): \nabla {\mathbf v} + \mathbb{F}_2^*\cc(y) \cdot {\mathbf v} \Big) + \mathbb{F}_3^*\cc(y) : \nabla {\mathbf v} + \mathbb{F}_4^*\cc(y) \mathbf{v} = \mathbf{0} \qquad \mbox{ in }\Omega,
\end{align}
where $\mathbb{F}_1^*\cc$ is a fourth order tensor; $\mathbb{F}_2^*\cc, \mathbb{F}_3^*\cc$ are third order tensors; $\mathbb{F}_4^*\cc$ is a second order tensor. They are given by the following formulae
\begin{equation}\label{eq:willis-push-forward-formula}
\begin{aligned}
\left[ \mathbb{F}^*_1\cc \right]_{ijkl}(y) & := \frac{1}{{\mathrm det }(D\mathbb{F})(x)} \sum_{p,q,r,s=1}^d \cc_{pqrs}(x) \frac{\partial \mathbb{F}_k}{\partial x_r}(x) \frac{\partial \mathbb{F}_l}{\partial x_s}(x) \frac{\partial \mathbb{F}_i}{\partial x_p}(x) \frac{\partial \mathbb{F}_j}{\partial x_q}(x),
\\
\left[ \mathbb{F}^*_2\cc \right]_{ijk}(y) & := \frac{1}{{\mathrm det }(D\mathbb{F})(x)} \sum_{p,q,r,s=1}^d \cc_{pqrs}(x) \frac{\partial^2 \mathbb{F}_k}{\partial x_r\partial x_s}(x) \frac{\partial \mathbb{F}_i}{\partial x_p}(x) \frac{\partial \mathbb{F}_j}{\partial x_q}(x),
\\
\left[ \mathbb{F}^*_3\cc \right]_{ijk}(y) & := \frac{1}{{\mathrm det }(D\mathbb{F})(x)} \sum_{p,q,r,s=1}^d \cc_{pqrs}(x) \frac{\partial \mathbb{F}_j}{\partial x_r}(x) \frac{\partial \mathbb{F}_k}{\partial x_s}(x) \frac{\partial^2 \mathbb{F}_i}{\partial x_p\partial x_q}(x),
\\
\left[ \mathbb{F}^*_4\cc \right]_{ij}(y) & := \frac{1}{{\mathrm det }(D\mathbb{F})(x)} \sum_{p,q,r,s=1}^d \cc_{pqrs}(x) \frac{\partial^2 \mathbb{F}_j}{\partial x_r\partial x_s}(x) \frac{\partial^2 \mathbb{F}_i}{\partial x_p \partial x_q}(x),
\end{aligned}
\end{equation}
for all $i,j,k,l\in\{1,\dots,d\}$, with the understanding that the right hand sides in \eqref{eq:willis-push-forward-formula} are computed at $x=\mathbb{F}^{-1}(y)$. Moreover we have,
\begin{align}
{\mathbf u}(x) = {\mathbf v}(x) \qquad \mbox{ for all }x\in\Omega\setminus B_2.
\end{align}
\end{proposition}

In (\ref{gauge}), the gauge is $\left[D\mathbb{F}^\top\right]^{-1}$. Remark that if one takes the gauge in (\ref{gauge}) to be the identity, then (\ref{eq:willis-model-change-of-variable-princ}) reduces to (\ref{cosserattransformed}) as shown in \cite{Norris_2011}. 
By abuse of language, we will refer to $\mathbb{F}^*\cc$ as the push-forward of $\cc$. Equation \eqref{eq:willis-model-change-of-variable-princ} is known as the Willis equation and can be achieved using perturbation techniques and variational principles in heterogeneous media \cite{Willis_1981}. Physically, the rank 3 tensors can be interpreted for instance in terms of a pre-stress and an effective body force due to the gradient of pre-stress \cite{Xiang_2016}.

\subsection{Kohn's transform and defect problems}
Following Kohn and co-authors \cite{Kohn_2008}, we fix a regularising parameter $\varepsilon>0$ and consider a Lipschitz map $\mathcal{F}_\varepsilon:\Omega\mapsto\Omega$ defined below
\begin{equation}\label{eq:std-near-cloak}
\mathcal{F}_\varepsilon(x) :=
\left\{
\begin{array}{cl}
x & \mbox{ for } x\in \Omega\setminus B_2,
\\[0.2 cm]
\left( \frac{2-2\varepsilon}{2-\varepsilon} + \frac{\vert x \vert}{2-\varepsilon}\right) \frac{x}{\vert x\vert} & \mbox{ for } x\in B_2 \setminus B_\varepsilon,
\\[0.2 cm]
\frac{x}{\varepsilon} & \mbox{ for }x\in B_\varepsilon.
\end{array}\right.
\end{equation}
Note that $\mathcal{F}_\varepsilon$ maps $B_\varepsilon$ onto $B_1$ and the annulus $B_2\setminus B_\varepsilon$ onto $B_2\setminus B_1$. Cloaking strategy with the above map corresponds to having $B_1$ as the cloaked region and the annulus $B_2\setminus B_1$ as the cloaking annulus. Remark that taking $\varepsilon=0$ in \eqref{eq:std-near-cloak} yields the map
\begin{equation}\label{eq:std-singular-cloak}
\mathcal{F}_0(x) :=
\left\{
\begin{array}{cl}
x & \mbox{ for } x\in \Omega\setminus B_2,
\\[0.2 cm]
\left( 1 + \frac12 \left\vert x \right\vert \right) \frac{x}{\vert x\vert} & \mbox{ for } x\in B_2 \setminus \{0\},
\end{array}\right.
\end{equation}
which is the singular transform of Greenleaf and coauthors \cite{greenleaf} \& Pendry and coauthors \cite{pendry}. The map $\mathcal{F}_0$ is smooth except at the point $0$. It maps $0$ to $B_1$ and $B_2\setminus\{0\}$ to $B_2\setminus B_1$.\\

\subsubsection{Construction in the Cosserat setting}
Given $0<\delta<1$, we take $\varepsilon=\sqrt[d]{\delta^2}$ and take the corresponding Kohn's regularised transform $\mathcal{F}_\varepsilon$ defined by \eqref{eq:std-near-cloak}. Then, our strategy to near-cloaking is to take $\mathbb{D}^\delta(y) = \mathcal{F}^\ast_\varepsilon\cc_0(y)$ for $y\in B_2\setminus B_1$ in the construction \eqref{eq:goal:cloak-defn} of the cloaking coefficient $\cc^\delta_\mathrm{cl}$. More precisely, we wish to consider
\begin{equation}\label{eq:cloak-coefficient-choice}
\cc^\delta_{\mathrm{cl}}(y) =
\left\{
\begin{array}{ll}
\cc_0 & \quad \mbox{ for }y\in\Omega\setminus B_2,\\[0.2 cm]
\mathcal{F}^\ast_\varepsilon\cc_0(y) & \quad \mbox{ for }y\in B_2\setminus B_1,\\[0.2 cm]
\cc_1 & \quad \mbox{ for }y\in B_1.
\end{array}\right.
\end{equation}
From the push-forward formula \eqref{eq:push-forward-formula}, it follows that the above cloaking coefficient is nothing but the push-forward of the following fourth order tensor
\begin{equation}\label{eq:cloak-coefficient-delta-equivalent}
\cc_\mathrm{cos}^\delta(x) =
\left\{
\begin{array}{ll}
\cc_0 & \quad \mbox{ for }x\in\Omega\setminus B_\varepsilon,\\[0.2 cm]
\varepsilon^{2-d}\cc_1\left(\frac{x}{\varepsilon}\right) & \quad \mbox{ for }x\in B_\varepsilon.
\end{array}\right.
\end{equation}
From the change-of-variable principle (Proposition \ref{prop:change-variable-cosserat}) for the Cosserat setting, it follows that in order to prove \eqref{eq:goal:distance-H12-norm}, it suffices to prove
\begin{align}\label{eq:delta-distance-H12-norm}
\left\Vert \mathbf{u}^\delta_\mathrm{cos} - \mathbf{u}_\mathrm{hom} \right\Vert_{[\mathrm{H}^\frac{1}{2}(\partial\Omega)]^d} \lesssim h\big(\delta\big),
\end{align}
where $\mathbf{u}^\delta_\mathrm{cos}(x)$ solves the following system with $\cc^\delta_\mathrm{cos}$ defined by \eqref{eq:cloak-coefficient-delta-equivalent} as the elasticity tensor:
\begin{equation}\label{eq:navier-delta-cosserat}
\left\{
\begin{aligned}
\nabla\cdot \left(\cc^\delta_\mathrm{cos}(x): \nabla \mathbf{u}^\delta_\mathrm{cos} \right) & = \mathbf{0} \qquad \mbox{ in }\Omega,
\\
\left(\cc^\delta_\mathrm{cos}(x) : \nabla {\mathbf u}^\delta_\mathrm{cos} \right){\mathbf n} & = \mathbf{g} \qquad \mbox{ on }\partial \Omega.
\end{aligned}\right.
\end{equation}
\subsubsection{Construction in the Willis setting}\label{ssec:willis-construction}
Given $0<\delta<1$, we take $\varepsilon=\sqrt[d]{\delta^2}$ and take the corresponding Kohn's regularised transform $\mathcal{F}_\varepsilon$ defined by \eqref{eq:std-near-cloak}. Then, our strategy to near-cloaking is to construct the following tensors using the push-forward formulae \eqref{eq:willis-push-forward-formula} for $y\in\Omega\setminus B_1$,
\begin{itemize}
\item a fourth order tensor $\mathbb{C}'(y):= \mathcal{F}^\ast_{\varepsilon 1}\mathbb{C}_0(y)$,
\item two third order tensors $\mathbb{D}'(y) := \mathcal{F}^\ast_{\varepsilon 2}\mathbb{C}_0(y)$, $\mathbb{S}'(y) := \mathcal{F}^\ast_{\varepsilon 3}\mathbb{C}_0(y)$,
\item a second order tensor $\mathbb{B}'(y):= \mathcal{F}^\ast_{\varepsilon 4}\mathbb{C}_0(y)$,
\end{itemize}
such that for $y\in B_1$, $\mathbb{C}'(y) = \cc_1$, $\mathbb{D}'=\mathbb{S}'=\mathbf{0}$, $\mathbb{B}'=\mathbf{0}$, followed by using these tensors in the following Willis system to solve for the unknown $\mathbf{u}^w_\mathrm{cl}$:
\begin{align}\label{eq:elasticity-system-willis}
\nabla \cdot \Big(\mathbb{C}'(y) : \nabla {\mathbf u}^w_\mathrm{cl} + \mathbb{D}'(y) \cdot {\mathbf u}^w_\mathrm{cl} \Big) + \mathbb{S}'(y) : \nabla {\mathbf u}^w_\mathrm{cl} + \mathbb{B}'(y) \mathbf{u}^w_\mathrm{cl} = \mathbf{0} \qquad \mbox{ in }\Omega.
\end{align}
From the push-forward formulae \eqref{eq:willis-push-forward-formula}, it follows that the above cloaking coefficients are nothing but the push-forwards of the following fourth order tensor:
\begin{equation}\label{eq:willis:cloak-coefficient-delta-equivalent}
\cc_\mathrm{wil}^\delta(x) =
\left\{
\begin{array}{ll}
\cc_0 & \quad \mbox{ for }x\in\Omega\setminus B_\varepsilon,\\[0.2 cm]
\varepsilon^{4-d}\cc_1\left(\frac{x}{\varepsilon}\right) & \quad \mbox{ for }x\in B_\varepsilon.
\end{array}\right.
\end{equation}
From the change-of-variable principle (Proposition \ref{prop:change-variable-willis}), it follows that in order to prove near-cloaking, it suffices to prove
\begin{align}\label{eq:willis-delta-distance-H12-norm}
\left\Vert \mathbf{u}^\delta_\mathrm{wil} - \mathbf{u}_\mathrm{hom} \right\Vert_{[\mathrm{H}^\frac{1}{2}(\partial\Omega)]^d} \lesssim h(\delta),
\end{align}
where $\mathbf{u}^\delta_\mathrm{wil}(x)$ solves the following system with $\cc^\delta_\mathrm{wil}$ defined by \eqref{eq:willis:cloak-coefficient-delta-equivalent} as the elasticity tensor:
\begin{equation}\label{eq:navier-delta}
\left\{
\begin{aligned}
\nabla\cdot \left(\cc^\delta_\mathrm{wil}(x): \nabla \mathbf{u}^\delta_\mathrm{wil} \right) & = \mathbf{0} \qquad \mbox{ in }\Omega,
\\
\left(\cc^\delta_\mathrm{wil}(x) : \nabla {\mathbf u}^\delta_\mathrm{wil} \right){\mathbf n} & = \mathbf{g} \qquad \mbox{ on }\partial \Omega.
\end{aligned}\right.
\end{equation}
Note that the elastic inclusions in the defect are scaled differently in Cosserat and in Willis settings. Note in particular, in dimension $d=2$, there is no high contrast in the elasticity inclusion which is in contrast with the Willis setting where there is a contrast of $\mathcal{O}(\varepsilon^2)$. Further note that application of Kohn's transform with the gauge (\ref{gauge}) to \eqref{eq:navier-delta} leads to \eqref{eq:elasticity-system-willis} with $\mathbf{g}$ as the traction datum, i.e.
$$\Big(\mathbb{C}'(y) : \nabla {\mathbf u}^w_\mathrm{cl}\Big)\mathbf{n}=\mathbf{g}\mbox{ on }\partial \Omega \; ,$$
as $\mathbb{D}'$ vanishes outside $B_2$.   

\section{Main results}\label{sec:near-cloaking-results}
Here we state and prove the main results of this paper. As explained earlier, the cloaking results -- both in the Cosserat and in the Willis settings -- essentially follow from the study of transmission elasticity systems of the type \eqref{eq:navier-transmission} with diametrically small inclusions. In \cite{Ammari_2013}, Ammari and co-authors provide an asymptotic expansion (a point-wise expression) for the solution to the transmission problem with small volume defect in terms of the solution to the homogeneous isotropic problem and in terms of the elastic moment tensors. We rephrase their result below.
\begin{theorem}\cite[Theorem 6.1 on page 4461]{Ammari_2013}\label{thm:Ammari-small-inclusion-transmission}
Let the spatial dimension be $d=2,3$. Let $\mathbf{u}_\mathrm{hom}(x)$ be the solution to the homogeneous isotropic system \eqref{eq:navier-isotropic} with Lam\'e parameters $\lambda_0$, $\mu_0$ satisfying the strong convexity condition (\ref{eq:strong-convexity-Lame-parameters}). Let $\mathbf{u}^\varepsilon(x)$ be the solution to the transmission problem \eqref{eq:navier-transmission} where the inclusion $D$ is of small diameter $\varepsilon$. Then, we have the estimate 
\begin{align}\label{eq:revise-new-thm}
\left\Vert \mathbf{u}_\mathrm{hom} - \mathbf{u}^\varepsilon \right\Vert_{[\mathrm H^\frac12(\partial\Omega)]^d} \le \, M\varepsilon^{2d-1}
\end{align}
for some constant $M$ independent of the Lam\'e parameters $\lambda_1$, $\mu_1$ in the inclusion $D$ as long as $\lambda_1\le \beta$ for some constant $\beta$. The constant $M$ being dependent on the upper bound $\beta$.
\end{theorem}
The readers are referred to \cite{Ammari_2013} for the proof of the above result. As it transpires from the proof therein, the constant $M$ appearing in the above Theorem \ref{thm:Ammari-small-inclusion-transmission} is a monotone increasing function of the bound $\beta$ on the bulk modulus in the inclusion. 
\begin{rmk}
A main handicap of Theorem \ref{thm:Ammari-small-inclusion-transmission}, as discussed in the introduction, is the imposition of a bound on the first Lam\'e parameter of the inclusion. A novelty of our work lies in lifting this constraint and improving the result of Theorem \ref{thm:Ammari-small-inclusion-transmission} to include arbitrary isotropic inclusions.
\end{rmk}
Before we state our results, we comment on the NtD maps associated with the transmission problem \eqref{eq:navier-transmission} and those associated with the soft and hard inclusion limits \eqref{eq:navier-transmission-soft-inclusion}-\eqref{eq:navier-transmission-hard-inclusion}. Let the inclusion $D$ in all the aforementioned transmission problems be $B_\varepsilon$. Furthermore, let us denote the solutions to \eqref{eq:navier-transmission}, \eqref{eq:navier-transmission-soft-inclusion} and \eqref{eq:navier-transmission-hard-inclusion} by $\mathbf{u}^\varepsilon$, $\mathbf{u}^\varepsilon_0$ and $\mathbf{u}^\varepsilon_\infty$ respectively. It follows from the ordering property of the NtD maps (see Lemma \ref{lem:ordering-NtD-maps}) that
\begin{align*}
\left\langle \mathbf{u}^\varepsilon_0, \mathbf{f} \right\rangle \le \left\langle \mathbf{u}^\varepsilon, \mathbf{f} \right\rangle \le \left\langle \mathbf{u}^\varepsilon_\infty, \mathbf{f} \right\rangle
\end{align*}
for any traction datum $\mathbf{f}\in [\mathrm H^{-\frac12}(\partial\Omega)]^d$. As the solution to the small volume defect problem is sandwiched between two functions, our approach is to show that these two extremes are close to the solution of the homogeneous isotropic system. This approach is inspired by an elegant idea in the work of Friedman and Vogelius \cite{Friedman_1989}. This strategy mimics the work \cite{Kohn_2008} where scalar equations are studied and which borrows ideas from \cite{Friedman_1989} dealing with the effect of extreme inhomogeneities in small inclusions on the boundary measurements. Here is our main result which generalises Theorem \ref{thm:Ammari-small-inclusion-transmission}.
\begin{theorem}\label{thm:revise}
Let the spatial dimension be $d=2,3$. Let $\mathbf{u}_\mathrm{hom}(x)$ be the solution to the homogeneous isotropic system \eqref{eq:navier-isotropic} with Lam\'e parameters $\lambda_0$, $\mu_0$ satisfying (\ref{eq:strong-convexity-Lame-parameters}). Let $\mathbf{u}^\varepsilon(x)$ be the solution to the transmission problem \eqref{eq:navier-transmission} where the inclusion $D$ is $B_\varepsilon$. Then, we have the estimate 
\begin{align}\label{eq:revise}
\left\Vert \mathbf{u}_\mathrm{hom} - \mathbf{u}^\varepsilon \right\Vert_{[\mathrm H^\frac12(\partial\Omega)]^d} \le \, M\varepsilon^{\frac{d}{2}}
\end{align}
where constant $M$ is independent of the Lam\'e parameters $\lambda_1$, $\mu_1$ in the inclusion.
\end{theorem}
\begin{rmk}
Note that, even though Theorem \ref{thm:revise} does not impose any constraint on the first Lam\'e parameter in the inclusion, the convergence rate is weaker than that in Theorem \ref{thm:Ammari-small-inclusion-transmission} as the result in \cite{Ammari_2013} is a point-wise result and is obtained via layer potential techniques as opposed to the variational approach employed in the present work.
\end{rmk}
From Theorem \ref{thm:Ammari-small-inclusion-transmission}, we have the following corollary:
\begin{cor}\label{cor:revise-1}
Let $\mathbf{u}_\mathrm{hom}(x)$ be the solution to the homogeneous isotropic system \eqref{eq:navier-isotropic} with Lam\'e parameters $\lambda_0$, $\mu_0$ satisfying (\ref{eq:strong-convexity-Lame-parameters}). Let $\mathbf{u}^\varepsilon_0(x)$ be the solution to the soft defect problem \eqref{eq:navier-transmission-soft-inclusion} where the defect $D$ is of small diameter $\varepsilon$. Then, we have the estimate 
\begin{align*}
\left\Vert \mathbf{u}_\mathrm{hom} - \mathbf{u}^\varepsilon_0 \right\Vert_{[\mathrm H^\frac12(\partial\Omega)]^d} \lesssim\, \varepsilon^{2d-1}
\end{align*}
\end{cor}
The above corollary simply follows by taking the following limit in the estimate \eqref{eq:revise-new-thm}:
\begin{align*}
\lim_{\max\{d\lambda_1+2\mu_1,2\mu_1 \}\to0} \left\Vert \mathbf{u}_\mathrm{hom} - \mathbf{u}^\varepsilon \right\Vert_{[\mathrm H^\frac12(\partial\Omega)]^d} \le \lim_{\max\{d\lambda_1+2\mu_1,2\mu_1 \}\to0} M\, \varepsilon^{2d-1}
\end{align*}
because the constant $M$ stays bounded in the above limit.\\
Our goal is now to prove the following result.
\begin{proposition}\label{prop:revise-2}
Let $\mathbf{u}_\mathrm{hom}(x)$ be the solution to the homogeneous isotropic system \eqref{eq:navier-isotropic} with Lam\'e parameters $\lambda_0$, $\mu_0$ satisfying (\ref{eq:strong-convexity-Lame-parameters}). Let $\mathbf{u}^\varepsilon_\infty(x)$ be the solution to the hard defect problem \eqref{eq:navier-transmission-hard-inclusion} where the defect $D$ is of small diameter $\varepsilon$. Then, we have the estimate 
\begin{align*}
\left\Vert \mathbf{u}_\mathrm{hom} - \mathbf{u}^\varepsilon_\infty \right\Vert_{[\mathrm H^\frac12(\partial\Omega)]^d} \le\, \varepsilon^{\frac{d}{2}}
\end{align*}
\end{proposition}
Before we give a proof of the above proposition, let us define the quotient space
\begin{align*}
\mathrm W(\Omega) := \mathrm H^1(\Omega)/ \mbox{Ker}\, {\mathbf e}(\cdot)
\end{align*}
as the space of classes of equivalence with respect to the relation 
\begin{align*}
{\mathbf u}\simeq {\mathbf v} \Longleftrightarrow {\mathbf u}-{\mathbf v} \in \mbox{Ker}\, {\mathbf e}(\cdot) \qquad \qquad \forall\, {\mathbf u},{\mathbf v}\in\mathrm H^1(\Omega).
\end{align*}
We denote by $\dot{\mathbf u}$ the class of equivalence represented by ${\mathbf u}$.\\
Let us recall the Korn's inequality in the quotient space $\mathrm W(\Omega)$: There exists a constant $C=C(\Omega)$ such that
\begin{align}
\left\Vert \dot{\mathbf v} \right\Vert_{\mathrm H^1(\Omega)} \le C \left\Vert {\mathbf e}(\dot{\mathbf v}) \right\Vert_{\mathrm L^2(\Omega)} \qquad \qquad \forall\, \dot{\mathbf v}\in \mathrm W(\Omega),
\end{align}
where 
\begin{align*}
\left\Vert \dot{\mathbf v} \right\Vert_{\mathrm H^1(\Omega)} := \inf_{{\mathbf r}\in\mbox{Ker}\, {\mathbf e}(\cdot)} \left\Vert {\mathbf v} + {\mathbf r} \right\Vert_{\mathrm H^1(\Omega)}
\quad
\mbox{ and }
\quad
{\mathbf e}(\dot{\mathbf v}) := {\mathbf e}({\mathbf w})\quad \mbox{ for any }{\mathbf w}\in\dot{\mathbf v}.
\end{align*}
We refer the readers to \cite{Ciarlet_2012} for further details on Korn's inequality and related topics. Thanks to the above Korn's inequality, we have the following norm on $\mathrm W(\Omega)$:
\begin{align}
\left\Vert \dot{\mathbf u} \right\Vert_{\mathrm W(\Omega)} = \left\Vert {\mathbf e}({\mathbf u}) \right\Vert_{\mathrm L^2(\Omega)} \qquad \forall\, {\mathbf u}\in\dot{\mathbf u}, \quad \dot{\mathbf u}\in\mathrm W(\Omega).
\end{align}

\begin{proof}[Proof of Proposition \ref{prop:revise-2}]
Note that the function space where we look for the solution to the hard inclusion problem \eqref{eq:navier-transmission-hard-inclusion} is the following
\begin{align}
\mathrm{W}^\varepsilon(\Omega) := \Big\{ \dot{\mathbf w}\in \mathrm W(\Omega) \mbox{ and }\dot{\mathbf w} = \dot{\mathbf 0}\, \mbox{ in }B_\varepsilon \Big\}.
\end{align}
From now onwards, we shall drop the $\dot{}$ notation for the representative of the equivalence class as long as there is no confusion.\\
Let us write the weak formulation for the homogeneous problem:
\begin{align}
\int_\Omega \big( \mathbb{C}_0:{\mathbf e}({\mathbf u}_{\mathrm{hom}}) \big) : {\mathbf e}({\mathbf w}) \, {\mathrm d}x = \int_{\partial\Omega} \mathbf{g}(x)\cdot {\mathbf w}(x)\, {\mathrm d}\sigma(x)
\qquad \forall \, {\mathbf w}\in \mathrm W(\Omega).
\end{align}
Next, we consider the weak formulation for the hard inclusion problem:
\begin{align}
\int_\Omega \big( \mathbb{C}_0:{\mathbf e}({\mathbf u}^\varepsilon_\infty) \big) : {\mathbf e}({\mathbf w}) \, {\mathrm d}x = \int_{\partial\Omega} \mathbf{g}(x)\cdot {\mathbf w}(x)\, {\mathrm d}\sigma(x)
\qquad \forall \, {\mathbf w}\in \mathrm{W}^\varepsilon(\Omega).
\end{align}
Hence we have
\begin{align}
\int_\Omega \big( \mathbb{C}_0:{\mathbf e}\left({\mathbf u}_{\mathrm{hom}} - {\mathbf u}^\varepsilon_\infty\right) \big) : {\mathbf e}({\mathbf w}) \, {\mathrm d}x = 0
\qquad \forall \, {\mathbf w}\in \mathrm{W}^\varepsilon(\Omega).
\end{align}
Because $\mathbb{C}_0$ is strongly convex and bounded, we have that ${\mathbf u}^\varepsilon_\infty$ is the projection of ${\mathbf u}_{\mathrm{hom}}$ onto the space $\mathrm{W}^\varepsilon$. Hence by Hilbert's theorem \cite{brezis}, we have
\begin{align*}
\left\Vert {\mathbf u}_{\mathrm{hom}} - {\mathbf u}^\varepsilon_\infty \right\Vert_{\mathrm W(\Omega)}
= \min_{{\mathbf v}\in \mathrm{W}^\varepsilon(\Omega)} \left\Vert {\mathbf u}_{\mathrm{hom}} - {\mathbf v} \right\Vert_{\mathrm W(\Omega)}
= \min_{{\mathbf v}\in \mathrm{W}^\varepsilon(\Omega)} \int_\Omega \left\vert {\mathbf e} \left( {\mathbf u}_{\mathrm{hom}} - {\mathbf v} \right)\right\vert^2\, {\mathrm d}x.
\end{align*}
So, the goal is to construct an element $\Theta\in\mathrm{W}^\varepsilon$ such that
\begin{align*}
\int_\Omega \left\vert {\mathbf e} \left( {\mathbf u}_{\mathrm{hom}} - \Theta \right)\right\vert^2\, {\mathrm d}x \le \varepsilon^d.
\end{align*}
To that end, let us define a smooth cut-off function
\begin{equation*}
\varphi(y)
=
\left\{
\begin{aligned}
1 & \qquad \mbox{ for }y\in B_1,
\\
0 & \qquad \mbox{ for }y\in \mathbb{R}^d\setminus B_2.
\end{aligned}
\right.
\end{equation*}
We rescale the above cut-off function as follows
\begin{equation*}
\varphi\left(\frac{x}{\varepsilon}\right)
=
\left\{
\begin{aligned}
1 & \qquad \mbox{ for }x\in B_\varepsilon,
\\
0 & \qquad \mbox{ for }x\in \mathbb{R}^d\setminus B_{2\varepsilon}.
\end{aligned}
\right.
\end{equation*}
We then define $\Theta\in\mathrm{W}^\varepsilon$ as follows
\begin{equation*}
\Theta(x) = \left( 1 - \varphi\left(\frac{x}{\varepsilon}\right) \right) {\mathbf u}_{\mathrm{hom}}(x).
\end{equation*}
Note that ${\mathbf u}_{\mathrm{hom}} - \Theta$ is supported in $B_{2\varepsilon}$. Let us compute
\begin{equation}\label{eq:sdp:proof-substitute}
\begin{aligned}
\int_\Omega \left\vert {\mathbf e} \left( {\mathbf u}_{\mathrm{hom}} - \Theta \right)\right\vert^2\, {\mathrm d}x
& =
\int\limits_{B_\varepsilon} \left\vert {\mathbf e} \left( {\mathbf u}_{\mathrm{hom}}(x)\right) \right\vert^2\, {\mathrm d}x 
+ \frac{1}{\varepsilon^2} \int\limits_{B_{2\varepsilon}\setminus B_\varepsilon} \left\vert \left[ \nabla \varphi \right]\left(\frac{x}{\varepsilon}\right) \right\vert^2  \left\vert {\mathbf u}_{\mathrm{hom}}(x) \right\vert^2\, {\mathrm d}x
\\
& \, \, + \int\limits_{B_{2\varepsilon}\setminus B_\varepsilon} \left\vert \varphi \left(\frac{x}{\varepsilon}\right) \right\vert^2  \left\vert {\mathbf e} \left( {\mathbf u}_{\mathrm{hom}}(x)\right) \right\vert^2\, {\mathrm d}x.
\end{aligned}
\end{equation}
Let us recall that ${\mathbf e}\left({\mathbf u}_{\mathrm{hom}}\right)$ belongs to $\mathrm L^\infty(\Omega)$ as it is a solution to the constant coefficient elliptic system, thanks to the elliptic regularity theory \cite{Agmon_1959, Agmon_1964} (see also \cite[section 4.7 on page 216]{Novotny_2004}). Furthermore, rescaling of the Korn's inequality yields
\begin{align*}
\int\limits_{B_{2\varepsilon}} \left\vert {\mathbf u}_{\mathrm{hom}}(x) \right\vert^2\, {\mathrm d}x \le \varepsilon^2 \int\limits_{B_{2\varepsilon}} \left\vert {\mathbf e} \left( {\mathbf u}_{\mathrm{hom}}(x)\right) \right\vert^2\, {\mathrm d}x.
\end{align*}
Hence all three terms on the right hand side of \eqref{eq:sdp:proof-substitute} are bounded by $\varepsilon^d$, thus proving the result.
\end{proof}

Now we state and prove our near-cloaking result in the Cosserat setting.
\begin{theorem}\label{thm:near-cloak-2D-cosserat}
Let the dimension $d=2,3$. Let $\mathbf{u}_\mathrm{hom}(x)$ be the solution to the homogeneous isotropic system \eqref{eq:navier-isotropic} with Lam\'e parameters $\lambda_0$, $\mu_0$ satisfying (\ref{eq:strong-convexity-Lame-parameters}). Given $0<\delta<1$, there exists a constant $M$ -- depending on $\Omega, \lambda_0, \mu_0$ -- such that
\begin{align}\label{eq:near-cloak-2D-cosserat}
\left\Vert \mathbf{u}^\delta_\mathrm{cl} - \mathbf{u}_\mathrm{hom} \right\Vert_{[\mathrm{H}^\frac{1}{2}(\partial\Omega)]^d} \le M\, \delta,
\end{align}
where $\mathbf{u}^\delta_\mathrm{cl}(x)$ solves the elasticity system \eqref{eq:navier-general} with elasticity tensor $\mathbb{C}^\delta_\mathrm{cl}(x)$ given by \eqref{eq:cloak-coefficient-choice}  with $\varepsilon=\sqrt[d]{\delta^2}$ and for any homogeneous isotropic tensor $\cc_1$ in the cloaked region $B_1$ with arbitrary Lam\'e parameters $\lambda_1$, $\mu_1$ subject to the strong convexity condition (\ref{eq:strong-convexity-Lame-parameters}).
%
\end{theorem}

\proof
Given $\delta$, take $\varepsilon = \sqrt[d]{\delta^2}$ and consider Kohn's regularised transform $\mathcal{F}_\varepsilon(x)$ defined by \eqref{eq:std-near-cloak}. From Proposition \ref{prop:change-variable-cosserat}, we know that proving the inequality \eqref{eq:near-cloak-2D-cosserat} is equivalent to proving \eqref{eq:delta-distance-H12-norm} where $\mathbf{u}^\delta_\mathrm{cos}(x)$ solves \eqref{eq:navier-delta-cosserat} with the elasticity tensor $\cc^\delta_\mathrm{cos}(x)$ defined by \eqref{eq:cloak-coefficient-delta-equivalent}. Note that the Lam\'e parameters $\lambda_1$ and $\mu_1$ in the inclusion $B_\varepsilon$ can be arbitrary in Theorem \ref{thm:revise}. Hence invoking Theorem \ref{thm:revise}, we have indeed proved the result.
\qed

Mutatis mutandis, the above proof applies in the following theorem for the Willis setting.

\begin{theorem}\label{thm:near-cloak-2D-3D-willis}
Let the dimension $d=2,3$. Let $\mathbf{u}_\mathrm{hom}(x)$ be the solution to the homogeneous isotropic system \eqref{eq:navier-isotropic} with Lam\'e parameters $\lambda_0$, $\mu_0$ satisfying (\ref{eq:strong-convexity-Lame-parameters}). Given $0<\delta<1$, there exists a constant $M$ -- depending on $\Omega, \lambda_0, \mu_0$ -- such that
\begin{align}\label{eq:near-cloak-2D-3D-willis}
\left\Vert \mathbf{u}^w_\mathrm{cl} - \mathbf{u}_\mathrm{hom} \right\Vert_{[\mathrm{H}^\frac{1}{2}(\partial\Omega)]^d} \le M\, \delta,
\end{align}
where $\mathbf{u}^w_\mathrm{cl}(x)$ solves the Willis system \eqref{eq:elasticity-system-willis} with the tensors $\mathbb{C}', \mathbb{D}', \mathbb{S}', \mathbb{B}'$ given in section \ref{sec:transformation-elasticity} for any homogeneous isotropic tensor $\cc_1$ in the cloaked region $B_1$ with arbitrary Lam\'e parameters $\lambda_1$, $\mu_1$ subject to the strong convexity condition (\ref{eq:strong-convexity-Lame-parameters}).
\end{theorem}

\section{Cloaking parameters and isotropic approximation}\label{sec:symmetrize}

\subsection{Cosserat and Willis elasticity parameters in polar coordinates}
First, we present the typical spatial distribution of the cloak parameters to illustrate the challenges faced by physicists and engineers in designing an elastodynamic cloak. We focus on the Cosserat case, and also give some results for Willis media,  
and consider Kohn's transform; we append the subscript cos to identify Cosserat, and the coefficient $\mathbb{C}_\mathrm{cos}$ depends on the regularising parameter $\varepsilon$ via the Lipschitz map $\mathcal{F}_\varepsilon$ in (\ref{eq:std-near-cloak}). 
 Consider the Lipschitz map $\mathcal{F}_\varepsilon:\Omega \mapsto \Omega$
\[
x := (x_1,x_2) \mapsto \left( \mathcal{F}_\varepsilon^{(1)}(x), \mathcal{F}_\varepsilon^{(2)}(x) \right) =: (y_1, y_2) = y.
\] 
If the cartesian coordinates $(x_1,x_2)$ are expressed in terms of  polar coordinates, $(r\cos\theta, r\sin\theta)$, then the new coordinates $(\mathcal{F}_\varepsilon^{(1)}(x), \mathcal{F}_\varepsilon^{(2)}(x))$ become $(r'\cos\theta, r'\sin\theta)$ with
\begin{equation}\label{eq:r-prime-defn}
r' :=
\left\{
\begin{array}{cl}
r & \mbox{ for } r \ge 2,
\\[0.2 cm]
\frac{2-2\varepsilon}{2-\varepsilon} + \frac{r}{2-\varepsilon} & \mbox{ for } \varepsilon < r < 2,
\\[0.2 cm]
\frac{r}{\varepsilon} & \mbox{ for } r \le \varepsilon.
\end{array}\right.
\end{equation}
We also recall that only the radial coordinate $r$ gets transformed by $\mathcal{F}_\varepsilon$ and the angular coordinate $\theta$ remains unchanged. Next, we reformulate the push-forward maps in terms of the polar coordinates.
The transformed elasticity tensor, $\mathbb{C}'_\mathrm{cos}$, has the following non-zero cylindrical components:
\begin{equation}
\begin{array}{ll}
C'_{\mathrm{cos}\, r'\!r'\!r'\!r'\!}\!=\!a(\varepsilon,r')(\lambda_0\!+\!2\mu_0), \!&
C'_{\mathrm{cos}\, \theta'\!\theta'\!\theta'\!\theta'}\!=\!\frac{1}{a(\varepsilon,r')}(\lambda_0\!+\!2\mu_0),\\[1.6mm]
C'_{\mathrm{cos}\, r'\!r'\!\theta'\!\theta'\!}=C'_{\mathrm{cos}\theta'\!\theta'\! r'\!r'}=\lambda_0, &
C'_{\mathrm{cos}\, r'\!\theta'\!\theta' \!r'}=C'_{\mathrm{cos}\theta'\! r'\!r'\!\theta'}=\mu_0, \\[1.6mm]
C'_{\mathrm{cos}\, r'\!\theta'\! r'\!\theta'}=a(\varepsilon,r') \mu_0, &
C'_{\mathrm{cos}\, \theta'\! r'\!\theta'\! r'}=\frac{1}{a(\varepsilon,r')} \mu_0,
\end{array}
\label{sc}
\end{equation}
in $B_2\setminus B_1$ while $\mathbb{C}'_{\mathrm{cos}}=\mathbb{C}_1$ in $B_1$ (i.e. inclusion) and $\mathbb{C}'_{\mathrm{cos}}=\mathbb{C}_0$ in $\Omega\setminus B_2$. For ease of notation, here $a(\varepsilon,r')$ stands for the following expression:
\[
a\left(\varepsilon,r'\right) := \frac{(\varepsilon-2)r'-2(\varepsilon-1)}{(\varepsilon-2)r'}
\qquad\qquad\qquad\mbox{ for }1\le r'\le 2.
\]
 Figure \ref{fig:num:polar-cloakCos} shows the variation of the components of the elasticity tensor $\mathbb{C}'_{\mathrm{cos}}(r')$ with respect to the radial variable.

We observe that when $\varepsilon$ gets smaller, the radial component of the elasticity tensor takes values very close to zero near the inner boundary of the cloak, which is unachievable in practice (the azimuthal component of the elasticity tensor being the inverse of the radial component, the elasticity tensor becomes extremely anisotropic). We note that $\mathbb{C}'_\mathrm{cos}$ does not have the minor symmetries, which is a further challenge for its fabrication. Therefore, manufacturing a metamaterial cloak would require a small enough value of $\varepsilon$ so that homogenization techniques can be applied to approximate the anisotropic elasticity tensor, and moreover one would need to achieve some substantial minor symmetry breaking in the elasticity tensor (so departing from classical homogenization results \cite{Camar-Eddine_2003}, possibly using some resonant elements that would make the cloaking interval narrowband at finite frequencies, in a way similar to that achieved with split ring resonators for electromagnetic waves \cite{schurig}). We will be revisiting the topic of approximate near-cloaking in the next subsection.
\begin{figure}[h!]
\mbox{}\hspace{-2cm}\includegraphics[width=16cm]{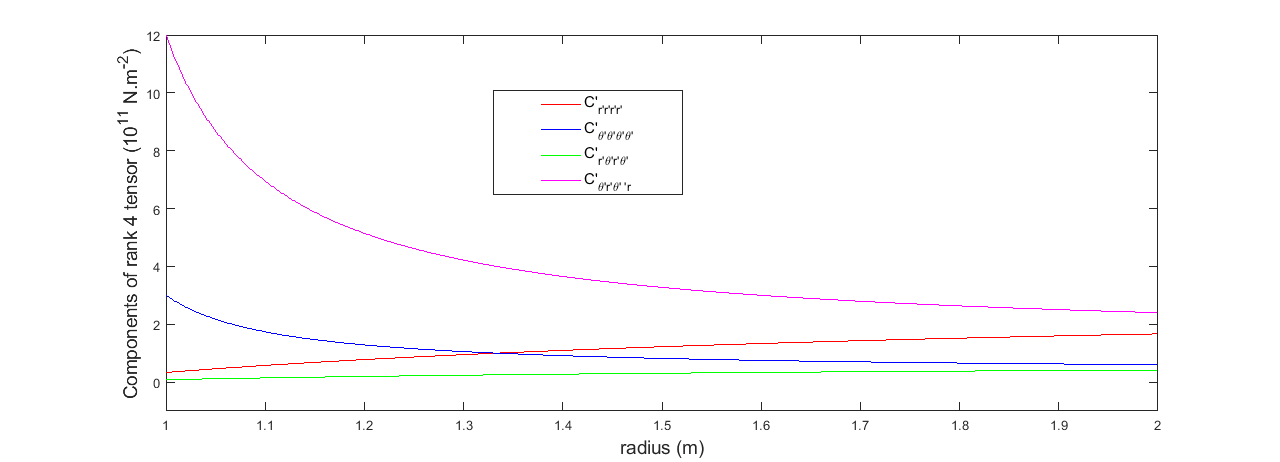}
\caption{The spatially varying components of the elasticity tensor $\mathbb{C}'_{\mathrm{cos}}(r')$ of a cylindrical Cosserat cloak from (\ref{sc}) for $r'\in[1,2]$ and $\varepsilon=0.2$ for a bulk medium with Lam\'e parameters $\lambda_0 = 1.5\times 10^{11}N.m^{-2}$ and $\mu_0 = 7.5\times 10^{10}N.m^{-2}$. One notes the substantial minor symmetry breaking for $r'=1$.}
\label{fig:num:polar-cloakCos}
\end{figure}

\begin{figure}[h!]
\mbox{}\hspace{-2cm}\includegraphics[width=16cm]{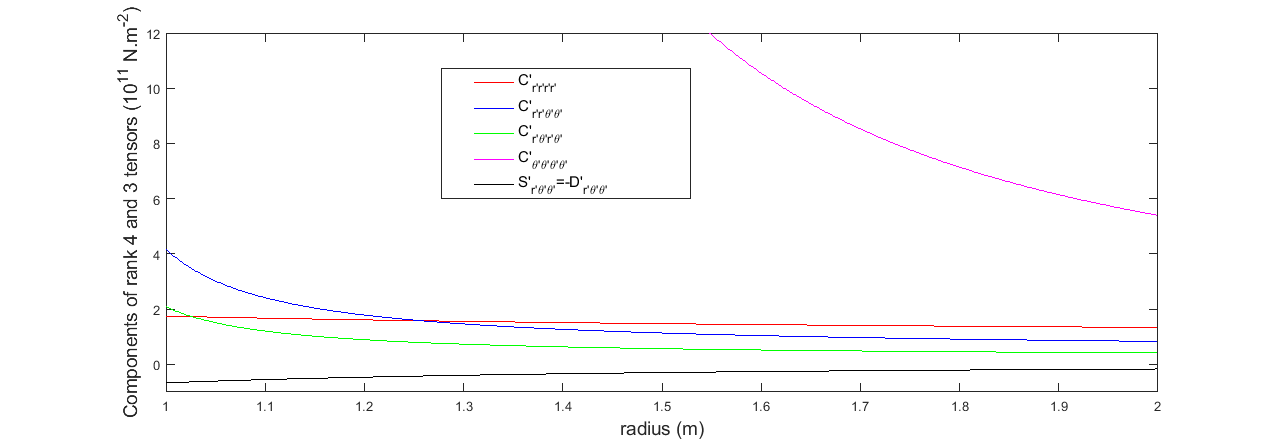}
\caption{The spatially varying components of the elasticity tensor $\mathbb{C}'_{\mathrm{will}}(r')$ and inertia tensors $\mathbb{S}'_{\mathrm{will}}(r')=-\mathbb{D}'_{\mathrm{will}}(r')$ of a cylindrical Willis cloak from (\ref{wil4}-\ref{wil5}) for $r'\in[1,2]$ and $\varepsilon=0.2$ for a bulk medium with Lam\'e parameters $\lambda_0 = 1.5\times 10^{11}N.m^{-2}$ and $\mu_0 = 7.5\times 10^{10}N.m^{-2}$. One notes the range of values for azimuthal component of $\mathbb{C}'_{\mathrm{will}}(r')$ ($[5.4\times10^{11},6.75\times10^{13}]$, out of scale) is larger than that of the
azimuthal component of $\mathbb{C}'_{\mathrm{cos}}(r')$ in Figure \ref{fig:num:polar-cloakCos}. However, $\mathbb{S}'_{\mathrm{will}}(r')$ and $\mathbb{D}'_{\mathrm{will}}(r')$ have relatively small values.}
\label{fig:num:polar-cloakWill}
\end{figure}

We now give the non-vanishing components of the rank 4 elasticity tensor in the Willis case, using the subscript wil to distinguish them from Cosserat:
\begin{equation}
\begin{array}{ll}
C'_{\mathrm{wil}\, r'\!r'\!r'\!r'\!}\!=\!\frac{r'(2-\varepsilon) + 2(\varepsilon-1)}{r'(2-\varepsilon)^3} \left( \lambda_0+2\, \mu_0 \right) , \!& \\
C'_{\mathrm{wil}\, \theta'\!\theta'\!\theta'\!\theta'}\!=\! \frac{\vert r'\vert^3(2-\varepsilon)}{{\left(r'(2-\varepsilon) + 2(\varepsilon-1)\right)}^3} \left( \lambda_0+2\, \mu_0 \right), \!& \\[1.6mm]
C'_{\mathrm{wil}\, r'\!r'\!\theta'\!\theta'\!}=C'_{\mathrm{wil}\,\theta'\!\theta'\! r'\!r'}=\frac{r'}{{(\varepsilon-2)}^2r'-2(2-\varepsilon)(\varepsilon-1)}\lambda_0, & \\
C'_{\mathrm{wil}\, r'\!\theta'\!\theta' \!r'}=C'_{\mathrm{wil}\, r'\!\theta'\! r'\!\theta'}=C'_{\mathrm{wil}\, \theta'\! r'\!\theta'\! r'}=C'_{\mathrm{wil}\, \theta'\! r'\!r'\!\theta'}
=\frac{r'}{{(\varepsilon-2)}^2r'-2(2-\varepsilon)(\varepsilon-1)}\mu_0, & \\[1.6mm]
\end{array}
\label{wil4}
\end{equation}
and the non-vanishing components of the rank 3 tensors
\begin{eqnarray}
\mathbb{S}'_{\mathrm{wil}\, r'\! r'\!r'}&=&-\frac{\lambda_0}{(r'(\varepsilon - 2) + 2 (1-\varepsilon)) (r'(\varepsilon - 2) - 2\varepsilon)}
 =-\mathbb{D}'_{\mathrm{wil}\, r'\! r'\!r'}\! ,\nonumber
\\
\mathbb{S}'_{\mathrm{wil}\, r'\theta'\!\theta'\!} &=&\mathbb{S}'_{\mathrm{wil}\, \theta'\!r' \!\theta'\!}\!=-\frac{-2\mu_0}{(r'(\varepsilon - 2) + 2 (1-\varepsilon))^2}
= -\mathbb{D}'_{\mathrm{wil}\,\theta'\!\theta'\!  r'}\!=-\mathbb{D}'_{\mathrm{wil}\, \theta'\! r'\theta'\!}\nonumber
\\
\mathbb{S}'_{\mathrm{wil}\,\theta'\!\theta'\!  r' \!} &=&
-\frac{(r'(\varepsilon - 2) - 2\varepsilon)(2\mu_0+\lambda_0)}{(r'(\varepsilon - 2) + 2 (1-\varepsilon))^3}
= -\mathbb{D}'_{\mathrm{wil}\,  r'\! \theta'\!\theta'\!}\; .\label{wil5}
\end{eqnarray}

In Figure \ref{fig:num:polar-cloakWill} we show the components of the elasticity tensor $\mathbb{C}'_{\mathrm{wil}}(r')$ and so we can compare and contrast the Cosserat and Willis media. Figures \ref{fig:num:polar-cloakCos} and \ref{fig:num:polar-cloakWill}, show that the required anisotropy of the elasticity tensor in the Willis cloak (evaluated by the ratio $C'_{\mathrm{wil}\theta'\!\theta'\!\theta'\!\theta'}(r')/C'_{\mathrm{wil}r'\!r'\!r'\!r'}(r')$) is much larger than that for the Cosserat cloak (evaluated by the ratio $C'_{\mathrm{cos}\theta'\!\theta'\!\theta'\!\theta'}(r')/C'_{\mathrm{cos}r'\!r'\!r'\!r'}(r')$), which is suggestive that a high contrast homogenization approach is needed for the Willis cloak. This could be circumvented, as the only non-vanishing component of the 3 tensors is relatively small compared with those of the elasticity tensor in Figure \ref{fig:num:polar-cloakWill}, by neglecting the 3 tensors to allow for a classical homogenization procedure or an higher order homogenization approach in the spirit of Shuvalov et al. \cite{shuvalov_2011}. This could have acceptable efficiency depending upon the applications one has in mind. For instance, two seismic cloak designs were proposed in \cite{diatta-achaoui-williscloak}, based upon the removal of 3 tensors in the transformed Willis equations (but keeping the density 2 tensor): For civil engineering applications, it is of foremost importance to fabricate a cloak which is broadband \cite{brule-handbook}, even if the cloaking efficiency is severely reduced.

\subsection{Symmetrized Cosserat cloak and its isotropic approximation}

The Cosserat cloak $\mathbb{C}'_{\mathrm{cos}}$ given in \eqref{sc} is positive definite. In fact, we have for any symmetric second order tensor $\xi=(\xi_{ij})$,
\begin{align*}
\left( \mathbb{C}'_{\mathrm{cos}} : \xi \right) : \xi & = a(\varepsilon,r')(\lambda_0+2\mu_0)\Big( \xi_{11}+\frac{\lambda_0}{a(\varepsilon,r')(\lambda_0+2\mu_0)}\xi_{22}\Big)^2 
\\
& \quad + \Big(2+a(\varepsilon, r')+\frac{1}{a(\varepsilon, r')}\Big)\mu_0\xi_{12}^2+ \frac{4\mu_0(\lambda_0+\mu_0)}{a(\varepsilon, r')(\lambda_0+2\mu_0)} \xi_{22}^2
\\
& \ge 0
\end{align*}
and that $\left( \mathbb{C}'_{\mathrm{cos}} : \xi \right) : \xi=0$ if and only if $\xi=0$.

Rather than developing an homogenization procedure to obtain an effective coefficient with no minor symmetries, we propose to symmetrise the Cosserat cloak in \eqref{sc} and then develop an approximation procedure through layered construction to approximate the aforementioned symmetrised tensor. Denoting the symmetrized Cosserat tensor by $\mathbb{C}^s_{\mathrm{cos}}$, our strategy is to take the same entries as in \eqref{sc} except for the the following four entries:
\[
C^s_{\mathrm{cos}\, r'\!\theta'\! r'\!\theta'} = C^s_{\mathrm{cos}\, r'\!\theta'\!\theta' \!r'} = C^s_{\mathrm{cos}\, \theta'\! r'\!\theta'\! r'} = C^s_{\mathrm{cos}\theta'\! r'\!r'\!\theta'} = \beta(r') \mu_0,
\]
where $\beta(r')$ is a positive function which needs to be chosen such that the minimum is attained in the following variational problem:
\begin{eqnarray}\label{eq:minimise-symmetry}
\min_{\xi\in\mathcal{M}}\{((\cc'_{\mathrm{cos}}-\cc^s_{\mathrm{cos}}):\xi):\xi\}
&=&\min_{\xi_{12}\in\R}\Big\{\Big( 2+a(\varepsilon,r')+\frac{1}{a(\varepsilon,r')}-4\beta(r')\mu\Big)\xi_{12}^2 \Big\},
\end{eqnarray}
where $\mathcal{M}$ stands for the space of real symmetric matrices. The equality in \eqref{eq:minimise-symmetry} is because
\begin{align*}
\left( \mathbb{C}^s_{\mathrm{cos}} : \xi \right) : \xi & = a(\varepsilon,r')(\lambda_0+2\mu_0)\Big( \xi_{11}+\frac{\lambda}{a(\varepsilon,r')(\lambda_0+2\mu_0)}\xi_{22}\Big)^2 
\\
& \quad + 4\beta(r')\mu\xi_{12}^2 + \frac{4\mu_0(\lambda_0+\mu_0)}{a(\varepsilon, r')(\lambda_0+2\mu_0)} \xi_{22}^2.
\end{align*}
Rather than addressing the minimisation problem in \eqref{eq:minimise-symmetry}, we single out a particular choice for $\beta$ which is easier to handle from a computational viewpoint. Note that there might be an optimal choice for the function $\beta(r')$ such that the minimum in the right hand side of \eqref{eq:minimise-symmetry} is a negative quantity. We, however, will not explore the optimal choice for $\beta$. We choose a $\beta$ such that the minimum in the right hand side of \eqref{eq:minimise-symmetry} is zero. This choice being
\[
\beta(r') = \frac{1}{4} \left( 2+a(\varepsilon,r')+\frac{1}{a(\varepsilon,r')} \right).
\]
Readers should take note that other choices of symmetrization for the Cosserat tensor have appeared in the literature \cite{Diatta-Kadic_2016, Sklan_2018}. Now that we have a symmetric, anisotropic and positive definite elasticity tensor $\mathbb{C}^s_{\mathrm{cos}}$, we can apply classical homogenization algorithms in order to achieve this tensor through isotropic layered media, for example. This is indeed the proposal of Greenleaf et al. \cite{greenleaf2008} (see also the recent work of Ghosh et al. \cite{ghosh}) in applying periodic homogenization arguments to design approximate multi-layered cloaks.
We consider an alternation of concentric layers of isotropic elasticity tensors $\cc^{(1)}$ and $\cc^{(2)}$ such that the effective elasticity tensor for this arrangement equates to the symmetrized Cosserat tensor $\mathbb{C}^s_{\mathrm{cos}}$. As hinted above, we will be considering periodic arrangement in radial direction (essentially, a one-dimensional periodic homogenization problem). Let $Y:=(0,1]$ denote the periodicity cell in the radial direction. Take $Y_1:=(0,\frac12]$ and $Y_2:=(\frac12,1]$. Denote by $\chi_1$ and $\chi_2$ their characteristic functions, i.e.
\begin{equation*}
\chi_i(\bar{r}) := \left\{
\begin{aligned}
1 & \qquad \mbox{ for }\bar{r}\in Y_i
\\
0 & \qquad \mbox{ otherwise}
\end{aligned}
\right.
\end{equation*}
for $i=1,2$. Our choice for the two-scale elasticity tensor is as follows:
\begin{equation}\label{eq:local-periodic-tensor}
\mathbb{C}(r',\bar{r}) := 
\left\{
\begin{aligned}
\mathbb{C}_0 & \qquad \mbox{ for }\vert r' \vert \le 1,
\\
\mathbb{C}^{(1)} \chi_1(\bar{r}) + \mathbb{C}^{(2)} \chi_2(\bar{r}) & \qquad \mbox{ for }1 < \vert r' \vert < 2,
\\
\mathbb{C}_1 & \qquad \mbox{ for }\vert r' \vert \ge 2,
\end{aligned}\right.
\end{equation}
with $\bar{r}\in Y$. Here the elasticity tensors $\cc_0$ and $\cc_1$ are homogeneous and isotropic of the form \eqref{eq:isotropic-elasticity-tensor} with Lam\'e parameters $\lambda_0,\mu_0$ and $\lambda_1, \mu_1$ respectively (both pairs satisfying the strong convexity condition \eqref{eq:strong-convexity-Lame-parameters}). The tensor $\cc_0$ corresponds to the background (i.e. $\Omega\setminus B_2$) and the tensor $\cc_1$ corresponds to the inclusion (i.e. $B_1$). The elasticity tensor in \eqref{eq:local-periodic-tensor} is nontrivially locally-periodic in the cloaking annulus (i.e. $B_2\setminus B_1$). The elasticity tensors $\mathbb{C}^{(1)}$ and $\mathbb{C}^{(2)}$ are isotropic (not necessarily homogeneous), i.e. their elements have the following structure:
\begin{align*}
C^{(\ell)}_{ijkl} = \lambda^{(\ell)}\, \delta_{ij}\delta_{kl} + \mu^{(\ell)} \left( \delta_{ik} \delta_{jl} + \delta_{il} \delta_{jk} \right),
\end{align*}
for $\ell=1,2$. Here $\lambda^{(\ell)}$ and $\mu^{(\ell)}$ are functions of the radial variable $r'$, to be determined later in this section. Introducing a small positive parameter $0<\tau\ll1$, we consider the following scaled elasticity tensors:
\[
\cc_\tau(r') := \cc \left( r' , \frac{r'}{\tau} \right)
\qquad \qquad \mbox{ for }\vert r' \vert \le 3.
\]
An interpretation of $tau$ is that  as it decreases, the layers become thinner, and the  number of layers increase in the annulus $B_2\setminus B_1$. Taking the scaled elasticity tensor, we shall consider the following boundary value problem:
\begin{equation*}
\left\{
\begin{aligned}
{\mathbf \nabla}\cdot \big(\cc_\tau : \nabla {\mathbf u}_\tau \big) & = \mathbf{0} \qquad \mbox{ in }\Omega,
\\
\big(\cc_\tau: \nabla {\mathbf u}_\tau\big)\mathbf{n} & = \mathbf{g} \qquad \mbox{ on }\partial \Omega.
\end{aligned}\right.
\end{equation*}
Periodic homogenization yields that ${\mathbf u}_\tau\approx {\mathbf u}^\ast$ in the $\tau\ll1$ regime where ${\mathbf u}^\ast$ solves the so-called homogenized equation:
\begin{equation*}
\left\{
\begin{aligned}
{\mathbf \nabla}\cdot \big(\cc^\ast : \nabla {\mathbf u}^\ast \big) & = \mathbf{0} \qquad \mbox{ in }\Omega,
\\
\big(\cc^\ast: \nabla {\mathbf u}^\ast\big)\mathbf{n} & = \mathbf{g} \qquad \mbox{ on }\partial \Omega,
\end{aligned}\right.
\end{equation*}
where the fourth order tensor $\cc^\ast$ is given as
\begin{equation*}
\mathbb{C}^\ast(r') := 
\left\{
\begin{aligned}
\mathbb{C}_0 & \qquad \mbox{ for }\vert r' \vert \le 1,
\\
\mathbb{C}^\ast & \qquad \mbox{ for }1 < \vert r' \vert < 2,
\\
\mathbb{C}_1 & \qquad \mbox{ for }\vert r' \vert \ge 2,
\end{aligned}\right.
\end{equation*}
where $\cc^\ast$ in the annulus $B_2\setminus B_1$ is given in terms of the Lam\'e parameters $\lambda^{(\ell)}$, $\mu^{(\ell)}$, $\ell=1,2$ as follows:
\begin{eqnarray}
C_{r'r'r'r'}^\ast &=& \frac{2\left(2\mu^{(1)}+\lambda^{(1)}\right)\left(2\mu^{(2)}+\lambda^{(2)}\right)}{2\mu^{(2)}+\lambda^{(2)}+2\mu^{(1)}+\lambda^{(1)}}; 
\nonumber\\
C_{r'\theta'r'\theta'}^\ast&=& \frac{2\mu^{(1)}\mu^{(2)}}{\mu^{(1)}+\mu^{(2)}};
\nonumber\\
C_{r'r'\theta'\theta'}^\ast &=& \frac{2\left(\mu^{(2)}\lambda^{(1)}+\lambda^{(1)}\lambda^{(2)}+\mu^{(1)}\lambda^{(2)}\right)}{2\mu^{(2)}+\lambda^{(2)}+2\mu^{(1)}+\lambda^{(1)}};
\nonumber\\
C_{\theta'\theta'\theta'\theta'}^\ast &=& \frac{2\left(2\mu^{(1)}\vert \lambda^{(1)}\vert^2\mu^{(2)}+\mu^{(1)}\vert\lambda^{(1)}\vert^2\lambda^{(2)}+2\vert\mu^{(1)}\vert^2\mu^{(2)}\lambda^{(1)}+\vert\mu^{(1)}\vert^2\lambda^{(1)}\lambda^{(2)}\right)}{\left(2\mu^{(1)}+\lambda^{(1)}\right)\left(2\mu^{(2)}+\lambda^{(2)}\right)}\nonumber\\
&+& \frac{2\left(2\mu^{(2)}\vert\lambda^{(2)}\vert^2\mu^{(1)}+\mu^{(2)}\vert\lambda^{(2)}\vert^2\lambda^{(1)}+2\vert\mu^{(2)}\vert^2\mu^{(1)}\lambda^{(2)}+\vert\mu^{(2)}\vert^2\lambda^{(1)}\lambda^{(2)}\right)}{\left(2\mu^{(1)}+\lambda^{(1)}\right)\left(2\mu^{(2)}+\lambda^{(2)}\right)}
\nonumber\\
&+&\frac{2\left(\mu^{(2)}\lambda^{(1)}+\lambda^{(1)}\lambda^{(2)}+\mu^{(1)}\lambda^{(2)}\right)^2}{\left(2\mu^{(2)}+\lambda^{(2)}+2\mu^{(1)}+\lambda^{(1)}\right)\left(2\mu^{(1)}+\lambda^{(1)}\right)\left(2\mu^{(2)}+\lambda^{(2)}\right)}.
\label{eq:backus}
\end{eqnarray}
In deriving these formulae one could refer to the effective formulas given by Backus \cite{backus} (see also the work of Patricio et al. \cite{patricio} on layered materials, and the books \cite{Allaire_2002, jikov} for a pedagogical exposition). 
The objective of this section is to derive expressions for $\lambda^{(\ell)}$ and $\mu^{(\ell)}$, $\ell=1,2$ such that
\[
\cc^\ast = \cc^s_{\mathrm{cos}},
\]
i.e. the homogenized tensor should coincide with the symmetrised Cosserat tensor. One can solve for the Lam\'e parameters to obtain the following expressions:
\begin{eqnarray}
\mu^{(1)}&=&\frac{\mu\mu^{(2)}\left(2a+\vert a\vert^2+1\right)}{8a\mu^{(2)}-2\mu a-\mu \vert a\vert^2-\mu}
\nonumber\\
\lambda^{(1)}&=& \frac{2\mu^{(2)}\left(6\mu \vert a\vert^2\mu^{(2)}-\lambda \mu \vert a\vert^2+4\lambda \vert a\vert^2\mu^{(2)}-4a\mu\mu^{(2)}-2a\lambda\mu -2\mu\mu^{(2)}-\lambda\mu\right)}{\left(-8a\mu^{(2)}+2a\mu +\mu \vert a\vert^2+\mu\right)\left(\lambda a+2\mu a-2\mu^{(2)}-\lambda\right)}
\nonumber\\
\lambda^{(2)}&=&2\mu^{(2)}\left(6\mu \vert a\vert^2 \mu^{(2)}+\lambda \mu \vert a\vert^2-4\mu a\mu^{(2)}+2\lambda\mu a-8\lambda a \mu^{(2)}-2\mu\mu^{(2)}+\lambda \mu+4\lambda \vert a\vert^2\mu^{(2)}\right)
\nonumber\\
& &\qquad \quad \times \Big(-14\mu \vert a\vert^2\mu^{(2)}+\lambda \mu \vert a\vert^2-\lambda \mu a+4\mu a\mu^{(2)}+8\lambda a\mu^{(2)}+2\mu \mu^{(2)}
\nonumber\\
& & \qquad \qquad \qquad   - \lambda\mu-8\lambda \vert a\vert^2\mu^{(2)}+\vert a\vert^3 \lambda \mu+4\vert \mu\vert^2 \vert a\vert^2+2\vert \mu\vert^2 \vert a\vert^3+2\vert\mu\vert^2 a\Big)^{-1};\label{eq:formulae-lame}
\end{eqnarray}
where $\mu^{(2)}$ is a solution of a degree $6$ polynomial with coefficients which are, themselves, polynomial combinations of $a,~\mu$ and $\lambda$. Note that all these Lam\'e parameters depend on the radial variable $r'$ and the regularization parameter $\varepsilon$ through the function $a(\varepsilon, r')$. To lighten the already daunting expressions in \eqref{eq:formulae-lame}, we do not make explicit the dependence upon this variable. In the next section (see Figure \ref{fg:plot}), we shall provide some numerical evidence in support of the isotropic approximation detailed here.

\section{Numerical simulations}\label{sec:numerics}
This section concerns some numerical tests that we have performed in support of the near-cloaking theory for the Navier systems in two spatial dimensions ($d=2$). Broadly speaking, our numerical tests concern the soft inclusion limit (Theorem \ref{thm:extreme-zero}), the hard inclusion limit (Theorem \ref{thm:extreme-infty}), the near-cloaking result in two dimensions and in the Cosserat setting (Theorem \ref{thm:near-cloak-2D-cosserat}) and some numerical tests in elastodynamics.

\subsection{Numerical tests in elastostatics}

We take the domain $\Omega$ to be a {\bf disk} of radius $10\, m$ and the inclusion $D$ to be one of the following: {\bf disk}, {\bf ellipse} and {\bf rectangle}.

For numerical convenience, we perform numerics on the Navier systems by providing Dirichlet data at the outer boundary $\partial\Omega$ rather than Neumann.
\begin{equation}\label{eq:numerics:navier-Dirichlet}
\left\{
\begin{aligned}
\nabla\cdot \big(\cc: \nabla {\mathbf u} \big) & = \mathbf{0} \qquad \mbox{ in }\Omega,
\\
{\mathbf u} & = \mathbf{f} \qquad \mbox{ on }\partial \Omega.
\end{aligned}\right.
\end{equation}
The Dirichlet datum that we have chosen is the following radial source:
\begin{align}\label{eq:numerics:radial-source}
\mathbf{f} = \left( \cos\theta, \sin\theta \right) \qquad \mbox{ on }\partial\Omega.
\end{align}
We have taken the following structure for the elasticity tensor $\cc$ in \eqref{eq:numerics:navier-Dirichlet}:
\begin{equation}\label{eq:numerics:elasticity-tensor}
\cc(x) =
\left\{
\begin{array}{ll}
\cc_0 & \quad \mbox{ for }x\in\Omega\setminus D,\\[0.2 cm]
\cc_1 & \quad \mbox{ for }x\in D.
\end{array}\right.
\end{equation}
Here $\cc_0$ and $\cc_1$ are isotropic homogeneous elasticity tensors with Lam\'e parameters $\lambda_0, \mu_0$ and $\lambda_1, \mu_1$ respectively. We choose the Lam\'e parameters in the background $\Omega\setminus D$ to be
\[
\lambda_0 = 1.5\times 10^{11}N.m^{-2}
\qquad
\qquad
\mu_0 = 7.5\times 10^{10}N.m^{-2}
\]
corresponding to Steel and we pick the Lam\'e parameters in the inclusion $D$ to be 
\[
\lambda_1 = 5.1\times 10^{10}N.m^{-2}
\qquad
\qquad
\mu_1 = 2.6\times 10^{10}N.m^{-2}
\]
corresponding to Aluminium.

The elasticity system \eqref{eq:numerics:navier-Dirichlet} with the above choice of isotropic background and isotropic inclusions is nothing but a transmission problem similar to \eqref{eq:navier-transmission} where one considers continuity of the displacement fields and the fluxes at the boundary of the inclusion $\partial D$. We solve \eqref{eq:numerics:navier-Dirichlet} for the unknown displacement field $\mathbf{u}(x)$ using finite elements in the COMSOL multiphysics package.

We start with a numerical study, which supports the theoretical findings of Theorem \ref{thm:revise} on the closeness of the solution $\mathbf{u}^\varepsilon$ of a small volume defect problem to the solution $\mathbf{u}_{\mathrm{hom}}$ of an unperturbed (homogeneous) problem. Our finite element computations for the quantification of the impact of a defect on an elastostatic field within a closed, constrained cavity, are shown in Table \ref{tab:disk:hardstuff}, and are in good agreement with the estimate
$
\left\Vert \mathbf{u}_\mathrm{hom} - \mathbf{u}^\varepsilon \right\Vert_{[\mathrm L^2(\partial\Omega)]^d} \le M\, \eta
$
which is valid for some constant $M$ independent of the Lam\'e parameters $\lambda_1$, $\mu_1$ in the inclusion $D$. We further show in Figure \ref{fg:plot} some plots of the displacement field in a homogeneous cavity, the same cavity with a defect, the same cavity with a defect of small diameter, the same cavity with a defect surrounded by an elastodynamic cloak (Cosserat), and the same cavity surrounded by 20 isotropic homogeneous layers approximating a symmetrized Cosserat cloak. We have numerically checked that the difference in $L^2$ norm on the boundary $\partial\Omega$ between the displacement field for the Cosserat cloak and the one for its approximate version decreases as it should when one increases the number of layers in the approximate cloak from $20$ to $40$. Further theoretical and numerical study is needed to check whether or not this difference behaves asymptotically like $1/N$, when the number of layers $N$ becomes large. In the conductivity case, it was numerically shown in \cite{petiteau14} that this is indeed the case, but the Cosserat cloak symmetrization might worsen the error estimate.

This illustrates the fact that the defect does not perturb the displacement field when it is either cloaked or small enough.

\begin{table}
\begin{tabular}{|l|l|l|l|c|c|c|c|c|c|r|r|r|r|}
   \hline
   &  \multicolumn{2}{|l}{Disks of decreasing area $\eta^2\pi$ with large contrast} & \\
   \hline
   &$\mu_1/\mu_0=\lambda_1/\lambda_0=10^2$ & $\mu_1/\mu_0=10^5$, $\lambda_1/\lambda_0=10^2$ & radius \\
   \hline
 & 5.2434752  & 6.8521241 &$\eta =1$ \\
${\Vert \mathbf{u}_\mathrm{hom} - \mathbf{u}_\mathrm{tr}\Vert}_{\mathrm L^2(\Omega\setminus B_1)}$ 
  & 1.1534421  & 1.2122752 & $\eta =10 ^{-1}$   \\
   & 1.123534 $\times 10^{-1}$  &  1.135342 $\times 10^{-1}$ & $\eta =10 ^{-2}$ \\
 & 1.12552 $\times 10^{-2}$  & 1.13777 $\times 10^{-2}$ & $\eta =10 ^{-3}$  \\
  & 1.1284 $\times 10^{-3}$  & 1.1391 $\times 10^{-3}$ & $\eta =10 ^{-4}$\\
  \hline
\end{tabular}
\caption{Distance  $\Vert \mathbf{u}_\mathrm{hom} - \mathbf{u}_\mathrm{tr} \Vert_{\mathrm L^2(\Omega\setminus B_1)}$  versus the radius of the small defect $\eta =10^{-m}$  $m=0,1,\cdots, 4,$  for  circular inclusions of large contrasts and decreasing area.}
\label{tab:disk:hardstuff}
\end{table}

\begin{figure}[htbp]
\resizebox{130mm}{!}{\includegraphics{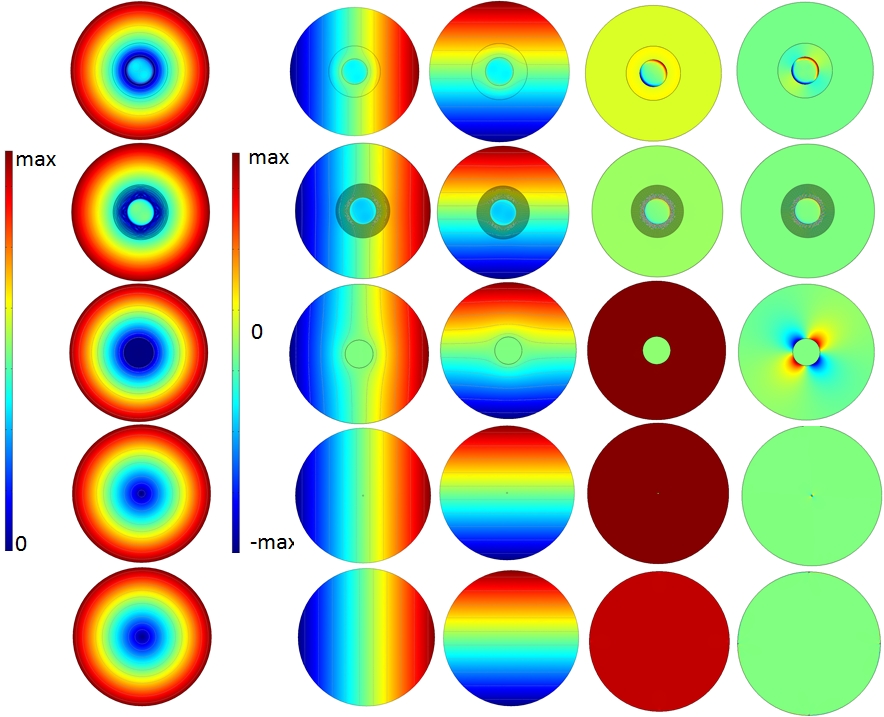}}
\caption{Plots of magnitude of displacement field $\sqrt{u_1^2+u_2^2}$ (first column), the horizontal displacement $u_1$ (second column), the vertical displacement $u_2$ (third column), the dilational strain $1/2(\partial u_1/\partial x_1+\partial u_2/\partial x_2)$ (fourth column) and the shear strain $1/2(\partial u_2/\partial x_1+\partial u_1/\partial x_2)$ (fifth column); First row: defect of radius $2\, m$ with Lam\'e parameters such that $\lambda_1/\lambda_0=10^2$ and $\mu_1/\mu_0=10^5$ surrounded by a Cosserat cloak of radii $2\, m$ and $4\, m$ in a cavity of radius $10\, m$ with an applied radial force on the boundary; Second row: same for an approximation of the Cosserat cloak with 20 isotropic homogeneous layers of identical thickness $0.1\, m$ with a symmetrized elasticity tensor as described in Section \ref{sec:symmetrize}; Third row: defect only; Fourth row: same for a small defect of radius $0.2\, m$; Fifth row: same without defect. Note we choose defects of radii $2\, m$ and $0.2\, m$ to enhance the discrepancies between plots.}
\label{fg:plot}
\end{figure}

To study the soft and hard inclusion limits, we will consider the following sequence of elasticity tensors
\begin{equation}\label{eq:numerics:elasticity-tensor-eta}
\cc^\eta(x) =
\left\{
\begin{array}{ll}
\cc_0 & \quad \mbox{ for }x\in\Omega\setminus D,\\[0.2 cm]
\eta\, \cc_1 & \quad \mbox{ for }x\in D,
\end{array}\right.
\end{equation}
with $\eta$ being a real parameter. We denote by $\mathbf{u}^\eta(x)$, the solution to the equation \eqref{eq:numerics:navier-Dirichlet} with elasticity tensor $\cc^\eta(x)$. We study the soft and hard inclusion limits while the parameters take extreme values (both in the regimes $\eta\ll1$ and $\eta\gg1$). The limit problems of interest are
\begin{equation}\label{eq:numerics:soft-hard-limits}
\left\{
\begin{aligned}
\nabla\cdot \big(\cc_0 : \nabla {\mathbf u}_0 \big) & = \mathbf{0} & \mbox{ in }\Omega\setminus D,
\\
\frac{\partial{\mathbf u}_0}{\partial {\mathbf m}}\Big|_+ & = \mathbf{0} & \mbox{ on }\partial D,
\\
\mathbf{u}_0 & = \mathbf{f} & \mbox{ on }\partial \Omega,
\end{aligned}\right.
\qquad \qquad 
\left\{
\begin{aligned}
\nabla\cdot \big(\cc_0 : \nabla {\mathbf u}_\infty \big) & = \mathbf{0} & \mbox{ in }\Omega\setminus D,
\\
{\mathbf u}_\infty & = \sum_{j=1}^{3} \alpha_j \pmb{\Psi}_j & \mbox{ on }\partial D,
\\
\mathbf{u}_\infty & = \mathbf{f} & \mbox{ on }\partial \Omega,
\end{aligned}\right.
\end{equation}
where $\{\pmb{\Psi}_j\}_{j=1}^3$ forms the basis functions in the space of rigid displacements in $D$. The coefficients $\alpha_j$ are determined by the orthogonality condition \eqref{eq:hard-inclusion-orthogonality-condition}. We numerically illustrate the asymptotic behaviour of $\left\Vert \mathbf{u}^\eta - \mathbf{u}_0 \right\Vert_{\mathrm L^2(\Omega\setminus D)}$ (resp. $\left\Vert \mathbf{u}^\eta - \mathbf{u}_\infty \right\Vert_{\mathrm L^2(\Omega\setminus D)}$) in terms of $\eta$ in the $\eta\ll1$ regime (resp. in the $\eta\gg1$ regime). We perform the asymptotic analysis on a series of geometric shapes while keeping their area constant and equal to $\pi$. This class of geometric shapes are (i) ellipses with varying semi-axes; (ii) rectangles of varying sides. The results are reported in tables and illustrated in graphs. We only compute the $\mathrm L^2$-norms of the difference as opposed to the $\mathrm H^1$-norms of Theorems \ref{thm:extreme-zero}, \ref{thm:extreme-infty}.

Computing solutions to \eqref{eq:numerics:navier-Dirichlet} in the regimes $\eta\ll1$ and $\eta\gg1$ corresponds to working with high-contrast problems and it is therefore necessary to mesh very finely the inclusion $D$ in order to achieve an accurate numerical solution. The case of hard inclusions is much more challenging numerically since the boundary condition at the inclusion for the limiting problem is unconventional in the sense that the Dirichlet data for $\mathbf{u}_\infty$ at the inner boundary $\partial D$ is given via the normalisation \eqref{eq:hard-inclusion-orthogonality-condition}. Regarding the mesh of the computational domain $\Omega$, we have numerically checked that all our results are fully converged, i.e., the numerical values for both $\left\Vert \mathbf{u}^\eta - \mathbf{u}_0 \right\Vert_{\mathrm L^2}$ and $\left\Vert \mathbf{u}^\eta - \mathbf{u}_\infty \right\Vert_{\mathrm L^2}$ do not change with further refinement in the mesh when we have about $840$ thousand triangles and $2$ thousand edges.\\

\item{\bf Ellipses:} 
We choose a series of ellipses of constant area $\pi,$ centered at the origin,  with semi-axes $a$ and $b$ taking the values $(a,b) =(1,1),$  $(2,1/2),$ $(4,1/4),$ $(7,1/7).$   The results in the soft inclusion regime are reported on Table \ref{tab:ellipses:soft:static} and in Figure \ref{fg:ellipses:soft:static}. We can deduce from the Log-Log plot in Figure \ref{fg:ellipses:soft:static} that
\begin{align}\label{eq:numerics:soft-asymptote}
\left\Vert \mathbf{u}^\eta - \mathbf{u}_0 \right\Vert_{\mathrm L^2(\Omega\setminus D)} \sim \mathcal{B}(D)\, \eta,
\end{align}
in the regime $\eta\ll1$, for some constant $\mathcal{B}(D)$ that depends on the shape of the elliptic inclusion $D$. The larger the eccentricity in the ellipse $D$, the larger the constant $\mathcal{B}(D)$ which can be up to $100$ times larger when the eccentricity is $7$ times larger.

The results in the case of hard inclusions are reported on Table \ref{tab:ellipses:hard:static} and in Figure \ref{fg:ellipses:hard:static}. We can infer from the Log-Log plot in Figure \ref{fg:ellipses:hard:static} that
\begin{align}\label{eq:numerics:hard-asymptote}
\left\Vert \mathbf{u}^\eta - \mathbf{u}_\infty \right\Vert_{\mathrm L^2(\Omega\setminus D)} \sim \mathcal{B}(D)\, \frac{1}{\eta},
\end{align}
in the regime $\eta\gg1$, for some constant $\mathcal{B}(D)$ that depends on the shape of the elliptic inclusion. 
\begin{table}[h!]
\centering
\begin{tabular}{|l|l|l|l|c|c|c|c|c|c|r|r|r|r|}
   \hline
   \multicolumn{4}{|l|}{$\left\Vert \mathbf{u}^\eta -  \mathbf{u}_0 \right\Vert_{\mathrm L^2(\Omega\setminus D)}$ for ellipses of constant area $\pi=  a\times b\times \pi$, ~ $b=1/a$}  &\\
   \hline
  $a=1$  & $a=2$ & $a=4$ & $a=7$ & Parameter $\eta$ \\
   \hline  
 6.94129  & 1.91765{$\times 10 ^{1}$} &  6.615782{$\times 10 ^{1}$}   & 1.2855071{$\times 10 ^{2}$} &$\eta =1$  \\
  1.25523  & 6.11473 & 3.717119{$\times 10 ^{1}$}   &9.224144{$\times 10 ^{1}$} & $\eta =10 ^{-1}$   \\
 1.3656{$\times 10 ^{-1}$}   & 7.8054{$\times 10 ^{-1}$} & 6.89548   & 2.41045{$\times 10 ^{1}$}&  $\eta =10 ^{-2}$ \\
1.378{$\times 10 ^{-2}$} &8.027{$\times 10 ^{-2}$} &   7.5399{$\times 10 ^{-1}$}  & 2.87416 &  $\eta =10 ^{-3}$  \\
 
1.38{$\times 10 ^{-3}$}  & 8.05{$\times 10 ^{-3}$} &   7.609{$\times 10 ^{-2}$}   & 2.9306{$\times 10 ^{-1}$} &  $\eta =10 ^{-4}$\\
 1.37903{$\times 10 ^{-4}$} & 8.06532{$\times 10 ^{-4}$}  &   7.6{$\times 10 ^{-3}$}  & 2.937{$\times 10 ^{-2}$} &  $\eta =10 ^{-5}$  \\
   \hline
\end{tabular}
\caption{Distance $\left\Vert \mathbf{u}^\eta - \mathbf{u}_0 \right\Vert_{\mathrm L^2(\Omega\setminus D)}$ versus the parameter $\eta =10^{-m}$ with $m=0,1,\dots, 5$ for soft elliptical inclusions of various semi-axes $a$, $b=\frac{1}{a}$ and constant area $\pi$.}
\label{tab:ellipses:soft:static}
\end{table}

\begin{figure}[htbp]
\resizebox{130mm}{!}{\includegraphics{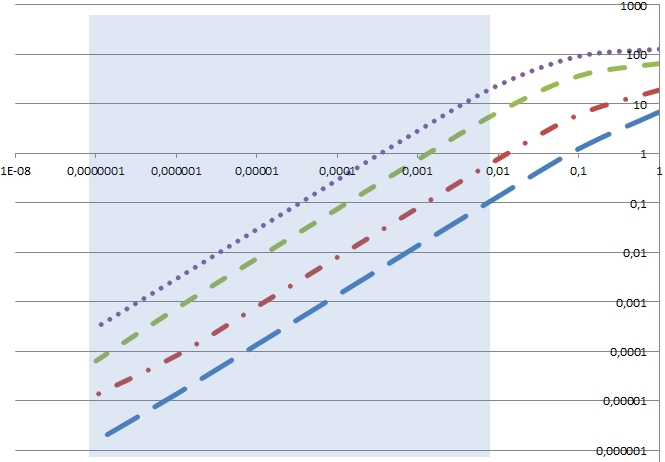}}
\caption{Curves, in Log-Log scale, of $\left\Vert \mathbf{u}^\eta -  \mathbf{u}_0 \right\Vert_{\mathrm L^2(\Omega\setminus D)}$ versus $\eta$, for $\eta =10^{-m}$ with $m=0,1,\dots,7$ in the cases where the inclusion is an ellipse of (constant) area $\pi$ and eccentricities $(1,1)$, $(2,\frac{1}{2})$, $(4,\frac{1}{4})$ and $(7,\frac{1}{7})$. The light blue shaded zone corresponds to the fully converged numerical values. The light blue shaded zone corresponds to fully converged numerical values. That is, when Log($\eta$) enters such a zone, the linear dependence of $\left\Vert \mathbf{u}^\eta -  \mathbf{u}_0 \right\Vert_{\mathrm L^2(\Omega\setminus D)}$ in $\eta$ holds.}
\label{fg:ellipses:soft:static}
\end{figure}

\begin{table}
\begin{tabular}{|l|l|l|l|c|c|c|c|c|c|r|r|r|r|}
   \hline
  \multicolumn{4}{|l|}{ $\left\Vert \mathbf{u}^\eta - \mathbf{u}_\infty \right\Vert_{\mathrm L^2(\Omega\setminus D)}$ for ellipses of constant area $\pi=  a\times b\times \pi$, ~ $b=1/a$}  &
\\
   \hline
  $a=1$  & $a=2$ & $a=4$ & $a=7$ & Parameter $\eta$
\\
   \hline  
  7.03416  &7.91058  &   1.15560{$\times 10 ^{1}$}  & 2.266599{$\times 10 ^{1}$}  & $\eta =1$ 
 \\
  1.28593   &  2.4304 &   9.60526   & 2.277612{$\times 10 ^{1}$} & $\eta =10 $  
 \\
 1.402{$\times 10 ^{-1}$}  &  3.8416{$\times 10 ^{-1}$} &   2.8943   & 1.308871{$\times 10 ^{1}$} & $\eta =10 ^{2}$
 \\
1.415{$\times 10 ^{-2}$}  & 4.072{$\times 10 ^{-2}$} &   3.6012{$\times 10 ^{-1}$}   & 2.4658 & $\eta =10 ^{3}$  
\\
 1.42{$\times 10 ^{-3}$}  & 4.1{$\times 10 ^{-3}$} &   3.691{$\times 10 ^{-2}$}  & 2.7046{$\times 10 ^{-1}$} & $\eta =10 ^{4}$
\\
  1.4165{\; $\times 10 ^{-4}$}  & 4.1111{\; $\times 10 ^{-4}$} &   3.7 $\times 10 ^{-3}$   & 2.731$\times 10 ^{-2}$ & $\eta =10 ^{5}$  \\
\hline
\end{tabular}
\caption{Distance $\left\Vert \mathbf{u}^\eta - \mathbf{u}_\infty \right\Vert_{\mathrm L^2(\Omega\setminus D)}$ versus the parameter $\eta =10^{m}$ with $m=0,1,\dots, 5$ for hard elliptical inclusions of various semi-axes $a$, $b=\frac{1}{a}$ and constant area $\pi$.}
\label{tab:ellipses:hard:static}
\end{table}

\begin{figure}[h!]
\resizebox{130mm}{!}{\includegraphics{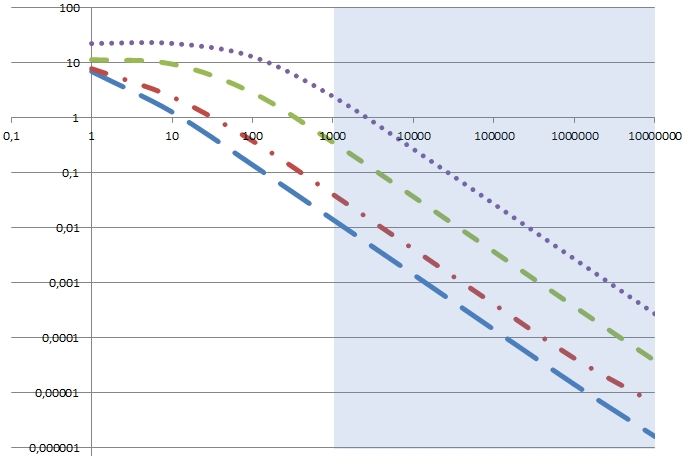}}
\caption{Curves, in Log-Log scale, of $\left\Vert \mathbf{u}^\eta -  \mathbf{u}_\infty \right\Vert_{\mathrm L^2(\Omega\setminus D)}$ versus $\eta$, for $\eta =10^{m}$ with $m=0,1,\dots,8$ in the cases where the inclusion is an ellipse of (constant) area $\pi$ and eccentricities $(1,1)$, $(2,\frac{1}{2})$, $(4,\frac{1}{4})$ and $(7,\frac{1}{7})$. The light blue shaded zone corresponds to the fully converged numerical values.}
\label{fg:ellipses:hard:static}
\end{figure}

In the work \cite{Ammari_2013}, the above convergence rates are given to be $\eta^{1/4}$ in the soft inclusion limit and to be $\eta^{-1/2}$ in the hard inclusion limit. However, the authors of \cite{Ammari_2013} do not claim their rates to be optimal. Our numerical tests suggest that there is room for improvement. Nevertheless, the distance was evaluated in the $\mathrm H^1$-norm in \cite{Ammari_2013} unlike the $\mathrm L^2$-norm computed above.

\item{\bf Rectangles}
Next we consider rectangular inclusions of constant area $\pi$, centred at the origin,  with sides $a$ and $b$ taking the values $(a,b) =(\sqrt{\pi},\sqrt{\pi}),$  $(2\sqrt{\pi},1/2\sqrt{\pi}),$ $(4\sqrt{\pi},1/4\sqrt{\pi}),$ $(7\sqrt{\pi},1/7\sqrt{\pi})$.  We report the results in Table \ref{tab:rectangles:soft:static}, Figure \ref{fg:rectangles:soft:static} for the soft inclusions ($\eta\ll1$) and in Table \ref{tab:rectangles:hard:static}, Figure \ref{fg:rectangles:hard:static} for the hard inclusions ($\eta\gg1$).
One can see that the asymptotic behaviour is similar to \eqref{eq:numerics:soft-asymptote} and \eqref{eq:numerics:hard-asymptote} where the constants $\mathcal{B}(D)$ depend upon the aspect ratio of the rectangular inclusion. The larger the aspect ratio, the larger the constant $\mathcal{B}(D)$. One also notes that although the area of the rectangular inclusions is kept to $\pi$ as for elliptical inclusions, the constant $\mathcal{B}(D)$ is slightly different than that for ellipses. This can be attributed to the fact that ellipses have a smooth boundary, unlike rectangles.

\begin{table}[h!]
\centering
\begin{tabular}{|l|l|l|l|c|c|c|c|c|c|r|r|r|r|}
   \hline
   \multicolumn{4}{|l|}{ $\left\Vert \mathbf{u}^\eta - \mathbf{u}_0 \right\Vert_{\mathrm L^2(\Omega\setminus D)}$
   for rectangles of constant area $\pi=  a\times b,$~ sides $b=\frac{\pi}{a}$}  &\\
   \hline
    $a=\sqrt{\pi}$   & $a=2\sqrt{\pi}$   &  $a=4\sqrt{\pi}$   & $a=7\sqrt{\pi}$ & Parameter $\eta$ \\
   \hline  
  8.95647  & 1.815529$\times 10 ^{1}$ &  5.759912$\times 10 ^{1}$   &         1.1612953$\times 10 ^{2}$   & $\eta =1$ 
 \\
 1.94399  & 5.69206 &   3.121681$\times 10 ^{1}$   &       8.19101$\times 10 ^{1}$   & $\eta =10 ^{-1}$   
\\
    2.2168$\times 10 ^{-1}$  &  7.2523$\times 10 ^{-1}$ &   5.61652   &          2.097353$\times 10 ^{-1}$       &  $\eta =10 ^{-2}$
 \\
2.25$\times 10 ^{-2}$  &  7.456$\times 10 ^{-2}$ & 6.1359$\times 10 ^{-1}$   &    2.48643 &  $\eta =10 ^{-3}$ 
 \\
  2.26$\times 10 ^{-3}$   &  7.47$\times 10 ^{-3}$  &   6.515$\times 10 ^{-2}$  & 2.5151$\times 10 ^{-1}$   &  $\eta =10 ^{-4}$
\\
   2.36725{\; $\times 10 ^{-4}$} & 7.43681{\; $\times 10 ^{-4}$} &   1.024$\times 10 ^{-2}$   & 2.355$\times 10 ^{-2}$ &  $\eta =10 ^{-5}$ 
 \\
  1.86009{\; $\times 10 ^{-5}$}  &  7.2795{\; $\times 10 ^{-5}$} &   5.03$\times 10 ^{-3}$   & 2.51$\times 10 ^{-3}$  &  $\eta =10 ^{-6}$  
\\
\hline
\end{tabular}
\caption{Distance $\left\Vert \mathbf{u}^\eta - \mathbf{u}_0 \right\Vert_{\mathrm L^2(\Omega\setminus D)}$ versus the parameter $\eta =10^{-m}$ with $m=0,1,\dots, 6$ for soft rectangular inclusions of sides $(\sqrt{\pi},\sqrt{\pi})$, $(2\sqrt{\pi},\frac{1}{2} \sqrt{\pi})$, $(4\sqrt{\pi},\frac{1}{4}\sqrt{\pi})$, $(7\sqrt{\pi},\frac{1}{7}\sqrt{\pi})$.}
\label{tab:rectangles:soft:static}
\end{table}

\begin{table} 
\begin{tabular}{|l|l|l|l|c|c|c|c|c|c|r|r|r|r|}
\hline
 \multicolumn{4}{|l|}{ $\left\Vert \mathbf{u}^\eta - \mathbf{u}_\infty \right\Vert_{\mathrm L^2(\Omega\setminus D)}$ for rectangles of constant area $\pi=  a\times b,$ ~$b=\frac{\pi}{a}$ } &
\\
   \hline
  $a=\sqrt{\pi}$  & $a=2\sqrt{\pi}$ &$a=4\sqrt{\pi}$ & $a=7\sqrt{\pi}$ & Parameter $\eta$
\\
   \hline  
7.22888  & 7.9371  &  1.089427$\times 10 ^{1}$  & 2.457099$\times 10 ^{1}$  &$\eta =1$ 
 \\
 1.37906   &  2.58855&  8.91507   & 2.252546$\times 10 ^{1}$ & $\eta =10 $  
 \\
 1.5574$\times 10 ^{-1}$  &  4.4344$\times 10 ^{-1}$&  2.63505  & 1.211387$\times 10 ^{1}$&  $\eta =10 ^{2}$
 \\
 1.579$\times 10 ^{-2}$  & 4.802$\times 10 ^{-2}$ &   3.2571$\times 10 ^{-1}$    & 2.13537 &  $\eta =10 ^{3}$  
\\
 1.58$\times 10 ^{-3}$  & 4.85$\times 10 ^{-3}$ &   3.325$\times 10 ^{-2}$  & 2.3179$\times 10 ^{-1}$ &  $\eta =10 ^{4}$
\\
 1.59758{$\times 10 ^{-4}$}  & 4.8702{$\times 10 ^{-4}$} &  3.24$\times 10 ^{-3}$   & 2.408$\times 10 ^{-2}$&  $\eta =10 ^{5}$ 
 \\
 2.36242{\; $\times 10 ^{-5}$}  & 5.09404{$\times 10 ^{-5}$} &   2.99417{$\times 10 ^{-4}$}  &  3.2$\times 10 ^{-3}$ &  $\eta =10 ^{6}$  
 \\
   \hline
\end{tabular}
\caption{Distance $\left\Vert \mathbf{u}^\eta - \mathbf{u}_\infty \right\Vert_{\mathrm L^2(\Omega\setminus D)}$ versus the parameter $\eta =10^{-m}$ with $m=0,1,\dots, 6$ for hard rectangular inclusions of sides $(\sqrt{\pi},\sqrt{\pi})$, $(2\sqrt{\pi},\frac{1}{2} \sqrt{\pi})$, $(4\sqrt{\pi},\frac{1}{4}\sqrt{\pi})$, $(7\sqrt{\pi},\frac{1}{7}\sqrt{\pi})$.}
\label{tab:rectangles:hard:static}
\end{table}

\begin{figure}[h!]
\resizebox{140mm}{!}{\includegraphics{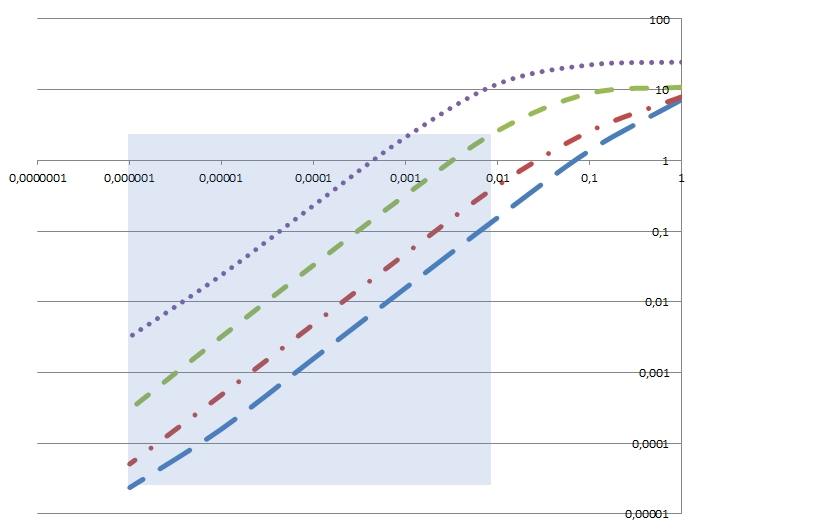}}
\caption{Curves, in Log-Log scale, of $\left\Vert \mathbf{u}^\eta -  \mathbf{u}_0 \right\Vert_{\mathrm L^2(\Omega\setminus D)}$ versus $\eta$, for $\eta =10^{-m}$ with $m=0,1,\dots,6$ in the cases where the inclusion is a rectangle of (constant) area $\pi$ and sides $(\sqrt{\pi},\sqrt{\pi})$, $(2\sqrt{\pi},\frac{1}{2} \sqrt{\pi})$, $(4\sqrt{\pi},\frac{1}{4}\sqrt{\pi})$, $(7\sqrt{\pi},\frac{1}{7}\sqrt{\pi})$. The light blue shaded zone corresponds to the fully converged numerical values.}
\label{fg:rectangles:soft:static}
\end{figure}

\begin{figure}[h!]
\resizebox{140mm}{!}{\includegraphics{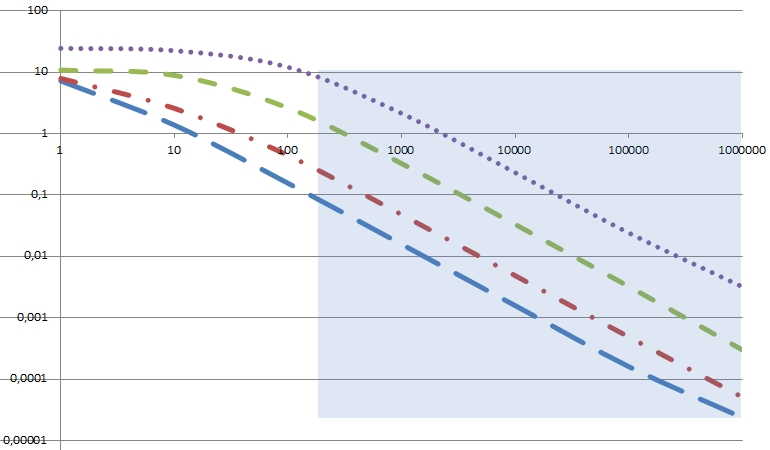}}
\caption{Curves, in $\log_{10}-\log_{10}$ scale, of $\left\Vert \mathbf{u}^\eta -  \mathbf{u}_\infty \right\Vert_{\mathrm L^2(\Omega\setminus D)}$ versus $\eta$, for $\eta =10^{m}$ with $m=0,1,\dots,6$ in the cases where the inclusion is a rectangle of (constant) area $\pi$ and sides $(\sqrt{\pi},\sqrt{\pi})$, $(2\sqrt{\pi},\frac{1}{2} \sqrt{\pi})$, $(4\sqrt{\pi},\frac{1}{4}\sqrt{\pi})$, $(7\sqrt{\pi},\frac{1}{7}\sqrt{\pi})$. The light blue shaded zone corresponds to the fully converged numerical values.}
\label{fg:rectangles:hard:static}
\end{figure}   

\subsection{Numerical tests in elastodynamics}
As opposed to the static case studied so far, we now briefly address the following time-harmonic Dirichlet elastodynamic problem at a non-zero frequency
\begin{equation}\label{eq:numerics:elasto-dynamics}
\left\{
\begin{aligned}
{\mathbf \nabla}\cdot \big(\cc^\eta: \nabla {\mathbf u}^\eta_q \big) + \rho\, \omega^2 {\mathbf u}^\eta_q & = \mathbf{0} \qquad \mbox{ in }\Omega,
\\
{\mathbf u}^\eta_q & = \mathbf{f} \qquad \mbox{ on }\partial \Omega,
\end{aligned}\right.
\end{equation}
for the unknown displacement field ${\mathbf u}^\eta_q:\Omega\to\R^d$. The fourth order elasticity tensor $\cc^\eta$ in \eqref{eq:numerics:elasto-dynamics} takes the form \eqref{eq:numerics:elasticity-tensor-eta} and $\rho, \omega$ denote the density and angular wave frequency respectively. The Dirichlet source ${\mathbf f}$ is taken to be radial and given by \eqref{eq:numerics:radial-source}. We study the soft inclusion limit, i.e., the $\eta\ll1$ regime for the elastodynamic problem \eqref{eq:numerics:elasto-dynamics}. The corresponding limit problem is as follows:
\begin{equation}\label{eq:numerics:limit-elasto-dynamics}
\left\{
\begin{aligned}
\nabla\cdot \big(\cc_0 : \nabla {\mathbf u}^0_q \big) + \rho\, \omega^2 {\mathbf u}^0_q & = \mathbf{0} & \mbox{ in }\Omega\setminus D,
\\
\frac{\partial{\mathbf u}^0_q}{\partial {\mathbf m}}\Big|_+ & = \mathbf{0} & \mbox{ on }\partial D,
\\
\mathbf{u}^0_q & = \mathbf{f} & \mbox{ on }\partial \Omega.
\end{aligned}\right.
\end{equation}
We have numerically studied the distance between the solution ${\mathbf u}^\eta_q$ to \eqref{eq:numerics:elasto-dynamics} and the solution ${\mathbf u}^0_q$ to \eqref{eq:numerics:limit-elasto-dynamics} in the $\mathrm L^2(\Omega\setminus D)$-norm.  We take the density to be $\rho_0=7.85 \times 10 ^{3} kg. m^{-3}$ in the background (steel) and $\rho_1 = 2.7 \times 10 ^{3} kg. m^{-3}$ in the inclusion (aluminium). The angular wave frequency $\omega$ is taken to be $0.1$ {rad.s$ ^{-1}$}. As in the elastostatics case, we have done these numerical experiments on elliptic inclusions with different eccentricities (while keeping the area fixed and equal to $\pi$) and on rectangular inclusions with different aspect ratios (again, while keeping the area fixed and equal to $\pi$). The results are tabulated in Table \ref{tab:ellipses:soft:dynamic} and Table \ref{tab:rectangles:soft:dynamic}. Clearly, the difference between the static and quasi-static cases is rather small (changes appear only in the third or fourth digit of decimal numbers).

\begin{table}
\begin{tabular}{|l|l|l|l|c|c|c|c|c|c|r|r|r|r|}
   \hline
  \multicolumn{4}{|l|}{$\Vert {\mathbf u^0_q}-  {\mathbf u^\eta_q}\Vert_{\mathrm L^2(\Omega\setminus D)}$  for ellipses of constant area $\pi=  a\times b\times \pi$}  & {$\omega=0.1$ rad.s$^{-1}$}\\
   \hline
 $a=1$  & $a=2$ & $a=4$ & $a=7$ & Parameter \\
   \hline  
 6.94118  & 1.917659$\times 10 ^{1}$ &  6.615783$\times 10 ^{1}$   & 1.2855079$\times 10 ^{2}$ &$\eta =1$  \\
  1.25521  & 6.11476 &  3.717121$\times 10 ^{1}$   &9.22415$\times 10 ^{1}$ & $\eta =10 ^{-1}$   \\
1.3656$\times 10 ^{-1}$   & 7.805$\times 10 ^{-1}$ & 6.8955   & 2.410452$\times 10 ^{1}$ &  $\eta =10 ^{-2}$ \\
1.378$\times 10 ^{-2}$ &8.022$\times 10 ^{-2}$ &   7.5401$\times 10 ^{-1}$  & 2.87416 &  $\eta =10 ^{-3}$  \\
1.38$\times 10 ^{-3}$  & 8$\times 10 ^{-3}$ &   7.611$\times 10 ^{-2}$   & 2.9305$\times 10 ^{-1}$ &  $\eta =10 ^{-4}$\\
1.38101{$\times 10 ^{-4}$}   & 7.57457{$\times 10 ^{-4}$}  &   7.61$\times 10 ^{-3}$   & 2.936$\times 10 ^{-2}$ &  $\eta =10 ^{-5}$  \\
\hline
\end{tabular}
\caption{Distance $\left\Vert \mathbf{u}^\eta_q - \mathbf{u}^0_q \right\Vert_{\mathrm L^2(\Omega\setminus D)}$ versus the parameter $\eta =10^{-m}$  $m=0,1,\cdots, 5,$  for soft elliptical inclusions of various semi-axes $a$, $b=\frac{1}{a}$ and constant area $\pi,$  at finite frequency $0.1$ rad.s$ ^{-1}$.}
 \label{tab:ellipses:soft:dynamic}
\end{table}

\begin{table}
\begin{tabular}{|l|l|l|l|c|c|c|c|c|c|r|r|r|r|}
   \hline
   \multicolumn{4}{|l|}{$\Vert {\mathbf u^0_q}-  {\mathbf u^\eta_q}\Vert_{\mathrm L^2(\Omega\setminus D)}$ for rectangles of constant area $\pi=  a\times b$} & $\omega=0.1$ rad.s$^{-1}$  \\
   \hline
  $a=\sqrt{\pi}$   & $a=2\sqrt{\pi}$   &  $a=4\sqrt{\pi}$   & $a=7\sqrt{\pi}$ & Parameter\\
   \hline  
 8.95647  & 1.816556$\times 10 ^{1}$ &  5.760182$\times 10 ^{1}$   &            1.1612956$\times 10 ^{2}$                     &$\eta =1$  \\
1.94707  & 5.69863 &   3.121703$\times 10 ^{1}$   &        8.191013$\times 10 ^{1}$           & $\eta =10 ^{-1}$   \\
 2.221$\times 10 ^{-1}$  &  7.2838$\times 10 ^{-1}$ &   5.611   &            2.097354$\times 10 ^{1}$         &  $\eta =10 ^{-2}$ \\
 2.253$\times 10 ^{-2}$  & 7.72$\times 10 ^{-2}$ &  6.0662$\times 10 ^{-1}$   &  2.48643 &  $\eta =10 ^{-3}$  \\
 2.25$\times 10 ^{-3}$  &  1.028$\times 10 ^{-2}$  &   5.792$\times 10 ^{-2}$   &   2.5151$\times 10 ^{-1}$ &  $\eta =10 ^{-4}$\\
 2.20882{$\times 10 ^{-4}$}  &  3.63$\times 10 ^{-3}$  &   3.43$\times 10 ^{-3}$   & 2.356$\times 10 ^{-2}$ &  $\eta =10 ^{-5}$  \\
   \hline
\end{tabular}
\caption{ Distance $\left\Vert \mathbf{u}^\eta_q - \mathbf{u}^0_q \right\Vert_{\mathrm L^2(\Omega\setminus D)}$ versus the parameter $\varepsilon =10^{-m}$  $m=0,1,\cdots, 5,$  for  soft  rectangular inclusions of various aspect ratios and constant area $\pi,$  at finite frequency $0.1$ rad.s$ ^{-1}$.}
\label{tab:rectangles:soft:dynamic}
\end{table}

\section{Concluding remarks and further discussions}

We close with a discussion of the array of open problems and extensions that arise: 
Our near-cloaking strategy carries through to anisotropic backgrounds with anisotropic inclusions, and this generalisation would be worthwhile, but more importantly our numerical experiments on the transmission problem (see Tables \ref{tab:ellipses:hard:static} and \ref{tab:rectangles:hard:static}) suggest sharper estimates in Theorem \ref{thm:extreme-infty} should hold; we encourage other analysts to explore the possibility of improving the estimates for elasticity equations with high contrast parameters. 

The problem of near-cloaking for the fully elastodynamic case (i.e. time dependent) is largely open. The extension of our work to the time domain could use similar tools to \cite{craster2017}, although the case of the time-dependent heat propagation equation uses results on the long time behaviour of solutions to parabolic problems, which do not translate straightforwardly to hyperbolic wave equations. To the best of our knowledge, there are no direct time-domain analysis of near cloaking results for wave equations (if the source oscillates periodically, then Fourier transform methods can yield near cloaking results in the time domain -- Nguyen et al. \cite{Nguyen_2012} employed a similar strategy for the scalar wave equation, albeit with lossy layers). Again, our results should encourage further work in this area. Another challenging open problem, of practical importance for physicists and engineers, is that of approximate near-cloaking, that is, designing a structured medium with effective properties achieving either the required asymmetry for elasticity tensor of the Cosserat cloak, or the fully symmetric rank-4, 3 and 2 tensors appearing in the Willis equation. In the former, as the transformed elasticity tensor does not have the minor symmetries, this means that we depart from the classical results of homogenization in elasticity \cite{Camar-Eddine_2003} and in the latter we may draw inspiration from the work of Milton and Willis \cite{Milton_2007} or the work of Shuvalov et al. \cite{shuvalov_2011}. We propose an approximate cloak design for the Cosserat cloak, based on an alternation of isotropic layers mimicking a symmetrized version of the elasticity tensor. Another issue of 
practical importance in earthquake engineering is to study the link between cloaking efficiency and the level of wave protection for an object surrounded by the cloak, as touched upon in \cite{diatta-guenneau-sphericalcloak, diatta-guenneau-Physics-D-2016} or to extend results to domains with irregular boundaries (e.g. sharp corners), as  discussed in \cite{Bigoni-polycloak} for antiplane shear waves and in \cite{Diattapolycloaks} for Electromagnetism.
 Clearly,  much remains to be investigated in elastic cloaking, near cloaking and approximate near cloaking.

\section*{ Acknowledgement}
A.D. and S.G. acknowledge European funding through the ERC Starting Grant ANAMORPHISM. H.H. and R.V.C. acknowledge the support of the EPSRC programme grant ``Mathematical fundamentals of Metamaterials for multiscale Physics and Mechanics'' (EP/L024926/1). R.V.C. also thanks the Leverhulme Trust for their support. S.G. acknowledges support from EPSRC as a named collaborator on grant EP/L024926/1.

\end{document}